\documentclass[3p]{elsarticle}

\usepackage{lineno,hyperref}

\modulolinenumbers[1]

\journal{arXiv}

\usepackage{amsfonts}
\usepackage{amsmath, amscd,amssymb,bm,amsbsy,epsf}
    \usepackage{enumerate}
    \usepackage{paralist}
\usepackage{bbm, dsfont}
\usepackage{graphicx,color}
\usepackage{subfigure}
\usepackage{verbatim}
\usepackage{inputenc}
\usepackage{bbm}
\usepackage{multicol}
\usepackage{balance}
\usepackage{verbatim}
\usepackage{bbm, doi}
\usepackage{sfmath} 
\usepackage{sans} 

\usepackage{multirow}
\usepackage{mathtools}

\DeclareMathAlphabet{\mathpzc}{OT1}{pzc}{m}{it}

\DeclareMathOperator*{\cl}{cl}

\DeclareMathOperator*{\tang}{tg}
\DeclareMathOperator*{\high}{\n}
\DeclareMathOperator*{\Div}{div}

\newtheorem{theorem}{Theorem}
\newtheorem{lemma}[theorem]{Lemma}
\newtheorem{corollary}[theorem]{Corollary}
\newtheorem{proposition}[theorem]{Proposition}
\newdefinition{definition}{Definition}
\newdefinition{hypothesis}{Hypothesis}
\newdefinition{remark}{Remark}
\newproof{proof}{Proof}
\newproof{sketch}{Sketch of the Proof}

\def\R{\bm{\mathbbm{R} } }
\def\N{\bm{\mathbbm{N} } }

\def\Hdiv{\mathbf{H}_{\Div}}

\def\eversor{\bm {\widehat{e} } }
\def\nuversor{\bm {\widehat{\nu} } }
\def\ind{\bm {\mathbbm{1} } }

\def\xthilde{\widetilde{\mathbf{x}}}
\def\pxthilde{(\widetilde{\mathbf{x}})}

\def\A{{\mathcal A}}
\def\B{{\mathcal B}}
\def\C{{\mathcal C}}
\def\D{\bm{D} }
\def\deps{\bm{D}^{\epsilon} }
\def\E{\mu}
\def\Eav{\bar{\mu}}
\def\Q{{\mathcal Q}}
\def\W{\mathbf W}
\def\X{\mathbf X}
\def\Y{\mathbf Y}
\def\0{\mathbf 0}
\def\f{\mathbf f}

\def\n{\bm {\widehat{n} } }
\def\u{\mathbf u}
\def\v{\mathbf v}
\def\w{\mathbf{w}}
\def\x{\mathbf x}
\def\y{\mathbf y}
\def\div{\boldsymbol{\nabla\cdot}}
\def\divt{\boldsymbol{\nabla}_{\!\!\scriptscriptstyle T}\boldsymbol{\cdot}}
\def\grad{\boldsymbol{\nabla}}
\def\gradt{\boldsymbol{\nabla}_{\!\!\scriptscriptstyle T}}
\def\bcol{\boldsymbol{:}}

\def\L2div{\mathbf{L^{2}_{div}}}
\def\H1bold{\mathbf{H^{1}}}
\def\vepsone{\mathbf{v}^{1,\epsilon}}
\def\vepstwo{\mathbf{v}^{\,2,\epsilon}}
\def\pepsone{p^{1,\epsilon}}
\def\pepstwo{p^{\,2,\epsilon}}
\def\p{\mathbf{p}}
\def\vtaneps{\mathbf{v}_{\scriptscriptstyle T}^{\,2,\epsilon}}
\def\vnormeps{\mathbf{v}_{\scriptscriptstyle N}^{\,2,\epsilon}}
\def\vtangeps{\mathbf{v}_{\,\tang}^{2,\epsilon}}
\def\vnormaleps{\mathbf{v}_{\,\n}^{2,\epsilon}}
\def\wtan{\mathbf{w}_{\scriptscriptstyle T}^{2}}
\def\wnorm{\mathbf{w}_{\scriptscriptstyle N}^{2}}

\def\wtangx{\mathbf{w}_{\,\tang (\widetilde{x})}^{2}}
\def\wnormalx{\mathbf{w}_{\,\n (\widetilde{x}\,)}^{2}}
\def\wtang{\mathbf{w}_{\,\tang}^{2}}
\def\wnormal{\mathbf{w}_{\,\n}^{2}}
\def\vone{\mathbf{v}^{1}}
\def\vtwo{\mathbf{v}^{2}}
\def\symgrad{\bm{\mathcal{E} } }
\def\Hpartial{H(\partial_z, \Omega_{2})}
\def\Hboldpartial{\mathbf{H}(\partial_z, \Omega_{2})}

\def\defining{\overset{\mathbf{def}} =}

\def\wone{\mathbf{w}^{1}}
\def\wtwo{\mathbf{w}^{2}}

\def\utwo{\mathbf{u}^{2}}
\def\vtang{\mathbf{v}_{\tang}^{\,2}}
\def\vnormal{\mathbf{v}_{\,\n}^{\,2}}

\def\pone{p^{1}}
\def\ptwo{p^{2}}
\def\veps{\mathbf{v}^{\,\epsilon}}
\def\peps{p^{\,\epsilon}}

\def\UxTtang{U^{T,\tang}(\xthilde)}
\def\UxTnormal{U^{T,\n}(\xthilde)}
\def\UxNtang{U^{N,\tang}(\xthilde)}
\def\UxNnormal{U^{N,\n}(\xthilde)}
\def\UTtang{U^{T,\tang}}
\def\UTnormal{U^{T,\n}}
\def\UNtang{U^{N,\tang}}
\def\UNnormal{U^{N,\n}}

\begin{document}

\begin{frontmatter}

\title{The asymptotic analysis of a Darcy-Stokes system\\
coupled through a curved interface}
\tnotetext[mytitlenote]{This material is based upon work supported by grant 98089 from the Department of Energy, Office of Science, USA and from project HERMES 27798 from Universidad Nacional de Colombia,
Sede Medell\'in.}


\author[mymainaddress]{Fernando A Morales}

\cortext[mycorrespondingauthor]{Corresponding Author}
\ead{famoralesj@unal.edu.co}

\address[mymainaddress]{Escuela de Matem\'aticas
Universidad Nacional de Colombia, Sede Medell\'in \\
Carrera 65 \# 59A--110 - Bloque 43, of 106,
Medell\'in - Colombia}


\begin{abstract}
The asymptotic analysis of a Darcy-Stokes system modeling the fluid exchange between a narrow channel (Stokes) and a porous medium (Darcy) coupled through a $ C^{2} $ curved interface, is presented. The channel is a cylindrical domain between the interface ($ \Gamma $) and a parallel translation of it ($ \Gamma + \epsilon \, \eversor_{N} $). The introduction of a change variable to fix the domain's geometry and the introduction of two systems of coordinates: the Cartesian and a local one (consistent with the geometry of the surface), permit to find a Darcy-Brinkman lower dimensional coupled system as the limiting form, when the width of the channel tends to zero ($ \epsilon \rightarrow 0 $). 
\end{abstract}

\end{frontmatter}



%
%

%
%
%
%
\section{Introduction}   \label{intro}
%
%
%
%
\noindent In this paper we continue the work presented in \cite{ShowMor17}, extending the result to a more general scenario. That is, we find the limiting form of a Darcy-Stokes (see \eqref{Eq Darcy-Stokes System} ) coupled system, within a saturated domain $ \Omega^{\epsilon} $ in $ \R^{N} $, consisting in three parts: a porous medium $ \Omega_{1} $ (Darcy flow), a narrow channel $ \Omega_{2} ^{\epsilon} $ whose width is of order $ \epsilon $ (Stokes flow) and a coupling interface $ \Gamma = \partial\Omega_{1} \cap \partial \Omega_{2}^{\epsilon} $, see Figure \ref{Fig Stream Lines} (a). In contrast with the system studied in \cite{ShowMor17}, where the interface is flat, here the analysis is extended to \textit{curved} interfaces. It will be seen that the limit is a fully-coupled system consisting of Darcy flow in the porous medium
$ \Omega_1 $ and a Brinkman-type flow on the part $ \Gamma $ of its boundary which now takes the form of a $ N-1 $ dimensional manifold.
\newline

\noindent The central motivation in looking for the limiting problem of our Darcy-Stokes system is to attain a new model free of the singularities present in \eqref{Eq Darcy-Stokes System}. These are the narrowness of the channel $ \mathcal{O}(\epsilon) $ and the high velocity of the fluid in the channel $ \mathcal{O}(\epsilon) $; both (geometry and velocity) with respect to the porous medium. Both singularities have substantial negative computational impact at the time of implementing the system, such as numerical instability and poor quality of the solutions. Moreover, when considering the case of curved interfaces, the geometry of the surface intensifies these effects, making even more relevant the search for an approximate singularity-free system as it is done here. 
\newline

\noindent The relevance of the Darcy-Stokes system itself, as well as its limiting form (a Darcy-Brinkman system) is confirmed by the numerous achievements reported in the literature: see \cite{ArbLehr}, \cite{Allaire2009}, \cite{CannonMeyer71} for the analytical approach, \cite{ArboBru}, \cite{GaticaBabuska2010}, \cite{Gatica2009},  \cite{JaffRob05} for the numerical analysis point of view, see \cite{Lamichhane}, \cite{XieXuXue} for numerical experimental coupling and \cite{Quarteroni_Brinkman2011} for a broad perspective and references. Moreover, the modeling and scaling of the problem have already been extensively justified in \cite{ShowMor17}, hence, this work is focused on addressing (rigorously) the interface geometry impact in the asymptotic analysis of the problem. It is important to consider the curvature of interfaces in the problem, rather than limiting the analysis to flat or periodic interfaces, because the fissures in a natural bedrock (where this phenomenon takes place) have wild geometry. In \cite{Dobberschutz2014}, \cite{Dobberschutz2015} the analysis is made using homogenisation techniques for periodically curved surfaces (which is the typical necessary assumption for this theory), in \cite{Neuss2000}, \cite{Neuss2001} the analysis is made using boundary layer techniques. However, no explicit results can be obtained, as usually with these methods. An early and simplified version of the present result can be found in \cite{Morales2}, where incorporating the interface geometry in the asymptotic analysis of a multiscale Darcy-Darcy coupled system is done and a explicit description of the limiting problem is given.
\newline

\noindent The successful analysis of the present work is due to keeping an interplay between two coordinate systems: the Cartesian and a \textit{local} one, consistent with the geometry of the interface $ \Gamma $. While it is convenient to handle the independent variables in Cartesian coordinates, the flow fields in the free fluid region $ \Omega_{2}^{\epsilon} $ are more manageable when decomposed in normal and tangential directions to the interface (the local system). The a-priori estimates, the properties of weak limits, as well as the structure of the limiting problem will be more easily derived with this \textit{double bookkeeping} of coordinate systems, rather than trying to leave behind one of them for good. It is therefore a \textit{strategic mistake} (not a mathematical one, of course) to seek for a transformation flattening out the interface, as it is the usual approach in traces' theory for Sobolev spaces. The proposed method is significantly simpler than other techniques and it is precisely this \textit{simplicity} which permits to obtain the limiting problem's explicit description for a problem of such complexity, as a multiscale Darcy-Stokes system.  
%
%
%
%
\begin{figure} 
	%
	\begin{subfigure}[Original Domain]
			{ \includegraphics[scale = 0.36]{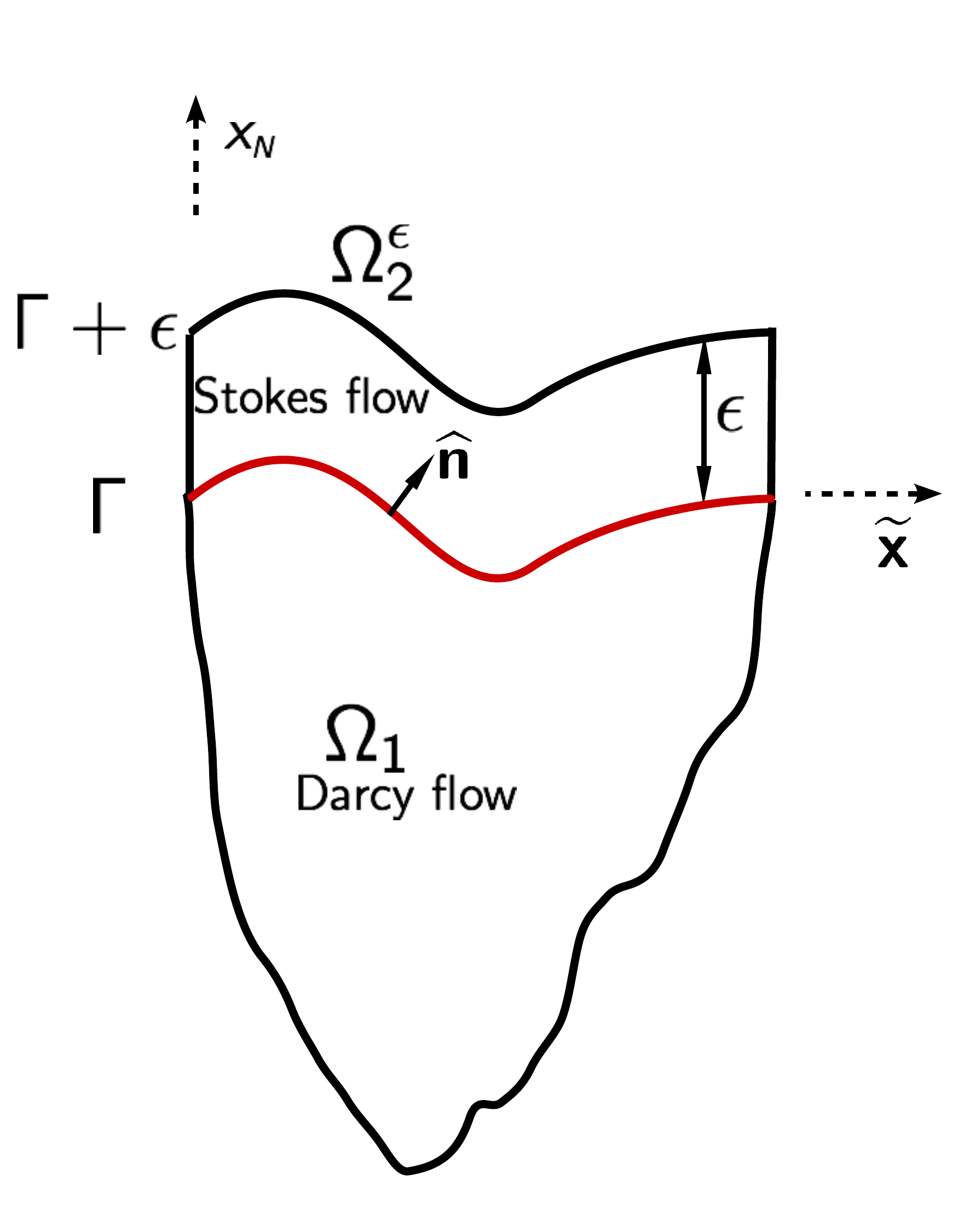} }
	\end{subfigure} 
	~ 
	\begin{subfigure}[Scaled Domain after the change of variable $ \x \mapsto \varphi(\x) $, with $ \varphi $ defined in Equation \eqref{Eq Map to Domain of Reference}.]
			{ \includegraphics[scale = 0.36]{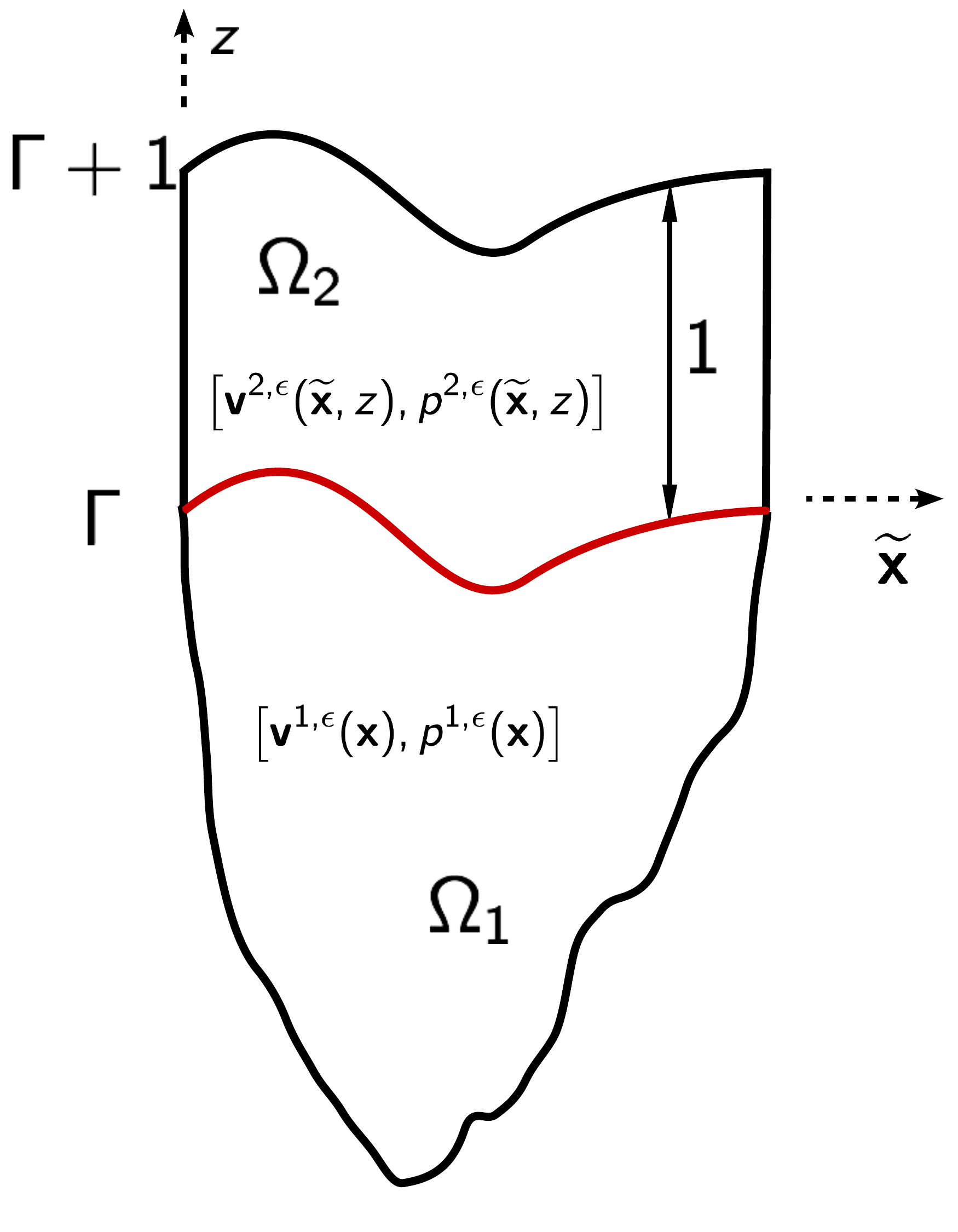} }             
	\end{subfigure} 
	\caption{Figure (a) depicts the original domain with a thin channel on top, where we set the Stokes flow. Figure (b) depicts the domain after scaling by the change of variables $ \x \mapsto \varphi(\x) $, where $ \varphi $ is defined in Equation \eqref{Eq Map to Domain of Reference}. This will be the domain of reference which is used for asymptotic analysis of the problem.}
	\label{Fig Stream Lines}
\end{figure}
%
%

%
%
\subsection*{Notation}
\noindent We shall use standard function spaces (see \cite{Temam79}, \cite{Adams}).
For any smooth bounded region $G$ in $\R^N$ with boundary $\partial
G$, the space of square integrable functions is denoted by $L^{2}(G)$,
and the Sobolev space $H^{1} (G) $ consists of those functions in
$L^2(G)$ for which each of the first-order weak partial derivatives
belongs to $L^2(G)$. The {\em trace} is the continuous linear
function $\gamma:H^{1}(G) \rightarrow L^{2}(\partial G)$ which agrees
with restriction to the boundary on smooth functions, i.e., $ \gamma(w) = w \big\vert_{\partial
G} $ if $ w \in C( \cl (G) ) $. Its kernel is $H_{0}^{1}(G) \defining \{w \in
H^{1}(G): \gamma(w) = 0\}$.  The trace space is $H^{1/2}(\partial G)
\defining \gamma(H^{1}(G))$, the range of $\gamma$ endowed with the
usual norm from the quotient space $H^{1}(G)/H_{0}^{1}(G)$, and we
denote by $H^{-1/2}(\partial G)$ its topological dual.
Column vectors and corresponding vector-valued functions will be
denoted by boldface symbols, e.g., we denote the product space
$\big[L^2(G)\big]^N$ by $\mathbf{L}^{2}(G)$ and the respective $ N $-tuple
of Sobolev spaces by $\mathbf{H}^{1}(G) \defining
\big[H^{1}(G)\big]^N$.  Each $w \in H^{1}(G)$ has {\em gradient}
$ \grad w = \big(\frac{\partial w}{\partial x_{1}}, \ldots,
\frac{\partial w}{\partial x_{\scriptscriptstyle N}} \big) \in
\mathbf{L}^{2}(G) $, furthermore we understand it as a row vector.
We shall also use the space $\Hdiv(G)$ of vector functions $\w \in
{\mathbf L}^2(G)$ whose weak divergence $\div \w$ belongs to
$L^{2}(G)$.
Let $\n$ be the unit outward normal vector on $\partial G$.  If $\w$
is a vector function on $\partial G$, we denote its normal component
by $ \w_{\,\n} = \gamma(\w)\cdot\n  $, its normal projection by $ \w(\n)  = \w_{\,\n} \, \n $. The tangential component is $ \w(\tang) = \w - \w(\n) $. The notation $ \w_{N} , \w_{T}$ indicate respectively, the last component and the first $ N - 1 $ components of the vector function $ \w $ in the canonical basis. 
For the functions $ \w \in \Hdiv(G) $, there is a {\em normal trace} defined on
the boundary values, which will be denoted by $\w \cdot \n \in
H^{-1/2}(\partial G)$. For those $\w \in \mathbf{H}^{1}(G)$ this
agrees with $\gamma(\w) \cdot \n$.
Greek letters are used to denote general second-order tensors.  The
contraction of two tensors is given by $ \sigma \bcol \kappa = \sum_{i, \,
j}\, \sigma_{ij}\kappa_{ij} $.
For a tensor-valued function $ \kappa $ on $ \partial G $, we denote the
normal component (vector) by $ \kappa(\n) \defining \sum_{j} \,
\kappa_{ij}\, \n_{j} \in \R^N $, and its normal and tangential parts by
$\kappa(\n)\cdot\n = \kappa(\n)_{\,\n} \defining \sum_{i, \, j} \,
\kappa_{ij} \n_{i} \n_{j}$ and $ \kappa(\n)_{\tang}
\defining \kappa (\n) - \kappa_{\,\n} \,\n$, respectively.
For a vector function $\w \in \mathbf{H}^{1}(G)$, the tensor $(\grad \w)_{ij} =
\frac{\partial w_i}{\partial x_j}$ is the {\em gradient} of $\w$ and
$\big(\symgrad(\w)\big)_{i j} = \tfrac{1}{2} \big(\frac{\partial w_i}{\partial x_j} +
\frac{\partial w_j}{\partial x_i} \big)$ is the {\em symmetric
gradient}.
\newline

\noindent For a column vector $\x =
\big(x_{\,1},\,\ldots,\,x_{\scriptscriptstyle N -
1},\,x_{\scriptscriptstyle N}\big)\in \R^{N}$ we denote the
corresponding vector in $\R^{N-1}$ consisting of the first $N-1$
components by $\widetilde{\x}=
\big(x_{1},\,\ldots,\,x_{\scriptscriptstyle N - 1}\big)$, and we
identify $\R^{N-1} \times \{0\}$ with $\R^{N-1}$ by $\x =
(\widetilde{\x},x_{\scriptscriptstyle N})$.
The operators $\gradt$,
$\gradt\boldsymbol{\cdot}$ denote respectively the
$\R^{N-1}$-gradient and the $\R^{N-1}$-divergence in the first $ N -1 $-canonical directions, i.e. $\gradt \defining \big(\frac{\partial}{\partial
x_{1}}, \ldots, \frac{\partial}{\partial x_{\scriptscriptstyle N - 1}}
\big)$, $ \gradt\boldsymbol{\cdot} \defining \big(\frac{\partial}{\partial
x_{1}}, \ldots, \frac{\partial}{\partial x_{\scriptscriptstyle N - 1}}
\big) $; moreover, we regard these operators as row vectors. Finally, $ \grad^{t}, \gradt^{t} $ denote the same operators written as column vectors, i.e., the operators denoted as column vectors. 
\begin{remark}\label{Rem Notation Vector vs Function}
It shall be noticed that different notations have been chosen to indicate the first $ N - 1 $ components: we use $ \xthilde $ for a vector as variable $ \x $, while we use $ \w_{T} $ for a vector function $ \w $ (or the operator $ \gradt, \grad $ ). This difference in notation will ease keeping track of the involved variables and will not introduce confusion.
\end{remark}
%
%
%
%
%
\subsection*{Preliminary Results}
\noindent We  close this section recalling some classic results.
\begin{lemma}\label{Th Surjectiveness from Hdiv to H^1/2}
Let $G \subset \R^{N}$ be an open set with Lipschitz boundary, let
$\n$ be the unit outward normal vector on $\partial G$.  The normal
trace operator $\u \in \Hdiv(G) \mapsto \u \cdot \n \in
H^{-1/2}(\partial G)$ is defined by
\begin{equation}\label{Eq Normal Trace Definition}
\big\langle \u\cdot \n, \phi \big\rangle
_{\scriptscriptstyle H^{-1/2}(\partial G), \, H^{1/2}(\partial G)}
\defining 
\int_{G} \Big(\u\cdot \grad \phi + \div \u \, \phi \Big)\, dx ,
\quad \phi \in H^{1}(G).
\end{equation}
For any $g\in H^{-1/2}(\partial G)$ there exists $\u\in \Hdiv(G)$ such
that $\u\cdot \n = g$ on $\partial G$ and $\Vert \u \Vert_{\Hdiv
(G)}\leq K \Vert g \Vert_{H^{-1/2}(\partial G)}$, with $K$ depending
only on the domain $G$. In particular, if $g$ belongs to
$L^{2}(\partial G)$, the function $\u$ satisfies the estimate $\Vert \u
\Vert_{\Hdiv (G)}\leq K \Vert g \Vert_{0,\partial G}$.
\end{lemma}
\begin{proof}
See Lemma 20.2 in \cite{TartarSobolev}.
\qed
\end{proof}
Next we recall a central result to be used in this work 

%
\begin{theorem}\label{Th well posedeness mixed formulation classic}
Consider the problem a pair satisfying
\begin{equation}\label{Pblm operators abstrac system}
\begin{split}
(\x, \y)\in \X\times \Y: \quad 
\A \x + \B '\y  = F_{1}\quad \text{in}\; \X ' , 
\\
- \B \x  + \C \y = F_{2} \quad \text{in}\; \Y ' .
\end{split}
\end{equation}
Here $ \X, \Y $, $\X ', \Y '$ are Hilbert spaces and their corresponding topological duals, $F_{1}\in \X '$, $F_{2} \in \Y '$ and the operators $\A: \X\rightarrow \X'$, $\B:
\X\rightarrow \Y '$, $\C: \Y\rightarrow \Y '$ are linear and continuous. Assume the operators satisfy
\begin{enumerate}[(i)]
\item $\A$ is non-negative and $\X$-coercive on $\ker (\B)$,

\item $\B$ satisfies the inf-sup condition 
\begin{equation}\label{Ineq general inf-sup condition}
   \inf_{\y \, \in \, \Y} \sup_{\x \, \in \, \X}
   \frac{\vert  \B\x(\y) \vert }{\Vert \x\Vert_{\X}\, \Vert \y \Vert_{\Y}}  >0\,,
\end{equation}

\item $C$ is non-negative and symmetric.
\end{enumerate}
Then, for every $F_{1} \in \X '$ and $F_{2} \in \Y '$ the problem
\eqref{Pblm operators abstrac system}
has a unique solution $(\x, \y)\in \X \times \Y$, and it satisfies the
estimate
\begin{equation} \label{mix-est}
\Vert\x\Vert_{\X} + \Vert \y\Vert_{\Y} \leq c\, (\Vert F_{1}\Vert_{\X '} 
+ \Vert F_{2}\Vert_{\Y '})
\end{equation}
for a positive constant $c$ depending only on the preceding
assumptions on $\A$, $\B$, and $\C$.
\end{theorem}
\begin{proof}
See Section 4 in \cite{GirRav79}. 
\qed
\end{proof}
%
%
%
\section{Geometric Setting and Formulation of the Problem}
\noindent In this section we introduce the Darcy-Stokes coupled system analogous to the one presented in \cite{ShowMor17}, for the case when the interface is curved. We begin with the geometric setting
%
\subsection{Geometric Setting and Change of Coordinates}\label{Sec Geometric Setting}
\noindent We describe here the geometry of the domains to be used in the present
work; see Figure \ref{Fig Stream Lines} (a) for the case $N = 2$. The disjoint bounded domains $\Omega_{1}$ and $\Omega_{2}^{\epsilon}$
in $\R^{N}$ share the common {\em interface}, $\Gamma \defining
\partial \Omega_{1} \cap \partial \Omega_{2}^{\epsilon} \subseteq \R^{N}$, and we define $\Omega^\epsilon \defining
\Omega_1 \cup \Gamma \cup \Omega_2^\epsilon$. The domain
$\Omega_{1}$ is the porous medium, and $\Omega_{2}^{\epsilon}$ is the
free fluid region. For simplicity we have assumed that the domain $\Omega_{2}^{\epsilon}$ is a cylinder defined by the interface $ \Gamma $ and a small height $ \epsilon > 0 $. It follows that the interface must verify specific requirements for a successful analysis
\begin{hypothesis}\label{Hyp Interface Geometric Conditions}
There exists $ G_{0}, G $ bounded open connected domains $ G_{0}, G\subset \R^{N-1} $, such that $ \cl (G)\subset G_0 $ and a $ C^{2}(G_0) $ function $ \zeta:G_0 \rightarrow \R $, such that the interface $ \Gamma $ can be described by
\begin{equation}\label{Def interface}
\Gamma\defining
\big\{\big(\xthilde,\,\zeta\,(\xthilde) \big):\,
\xthilde\in G\,\big\} ,
\end{equation}
i.e., $ \Gamma $ is a $ N-1 $ manifold in $ \R^{N} $. The domain $ \Omega_{2}^{\epsilon} $ is described by
\begin{equation}\label{omega 2 epsilon}
\Omega_2^{\epsilon}\defining
\big\{\big(\xthilde,\,y\big):
\zeta\,(\xthilde)< y< \zeta\,(\xthilde)+\epsilon,\,
\xthilde\in G \big\} ,
\end{equation}
\end{hypothesis}
\begin{remark}\label{Rem Interface Geometric Conditions}
\begin{enumerate}[(i)]
\item Observe that the domain $ G $ is the orthogonal projection of the open surface $ \Gamma\subseteq \R^{N} $ into $ \R^{N-1} $. 

\item Notice that due to the properties of $ \zeta $ it must hold that if $ \n = \n\pxthilde $ is the upwards unitary vector, orthogonal to the surface $ \Gamma $ then
\begin{equation}\label{Eq constraint on the normal vector}
\delta\defining\inf
\big\{\n\pxthilde\cdot\eversor_{N}:\xthilde\in G\big\}>0 .
\end{equation}
\end{enumerate}
\noindent For simplicity of notation in the following we write
\begin{enumerate}[(i)]
\setcounter{enumi}{2}
\item 
\begin{equation}\label{Def epsilon top surface}
\Gamma+\epsilon\defining
\big\{\big(\xthilde,\,\zeta\,(\xthilde)+\epsilon\big):\,
\xthilde\in G\,\big\} ,
\end{equation}
\item 
\begin{align}\label{Def Reference Domain Notational Simplicity}
& \Omega_{2} \defining \Omega_{2}^{1}, &
& \Omega \defining \Omega^{1} .
\end{align}
\end{enumerate}
\end{remark}
For the asymptotic analysis of the coupled system, a domain of reference $ \Omega $ will have to be fixed, see Figure \ref{Fig Stream Lines} (b). Therefore we adopt a bijection between domains and account for the changes in the differential operators.
\begin{definition}\label{Def Map to Domain of Reference}
Let $ \varphi: \Omega_{2}^{\epsilon} \rightarrow \Omega$ be the change of variables defined by
\begin{equation}\label{Eq Map to Domain of Reference}
\varphi(y_{1}, \ldots, y_{N -1}, y_{N}) \defining
\begin{Bmatrix}
y_{1} \\
\vdots\\
y_{N -1} \\
\epsilon^{-1} \big(y_{N} - \zeta(y_{1}, \ldots, y_{N-1})\big)
+ \zeta\big(y_{1}, \ldots, y_{N-1}\big) 
\end{Bmatrix} ,
\end{equation}
%
with $ \y = \left(y_{1},\,\ldots ,\,y_{N-1},\,y_{N} \right) \in \Omega_{2}^{\epsilon} $. Also, denote $ \x = (x_{1}, \ldots, x_{N - 1},  z) \defining \varphi(\y) $, i.e.
\begin{align}\label{Def Name of Variables}
& \x = (x_{1}, \ldots, x_{N - 1},  z) \in \Omega_{2}, &
& \x \cdot \eversor_{\ell} = \varphi(\y)\cdot \eversor_{\ell} .
\end{align}
\end{definition}
\begin{remark}\label{Rem Domains Bijective Map}
Observe that $ \varphi: \Omega_{2}^{\epsilon} \rightarrow \Omega_{2} $ is a bijective map, see Figure \ref{Fig Stream Lines} (b).
\end{remark}
%
%
\subsubsection*{Gradient Operator} 
\noindent Denote by $ \prescript{\y}{}{\grad} , \prescript{\x}{}{\grad} $ the gradient operators with respect to the variables $ \y $ and $ \x $ respectively. Due to Definition \ref{Def Reference Domain Notational Simplicity} above, a direct computation shows that these operators satisfy the relationship
\begin{equation}\label{Eq gradient structure for scalar function}
\prescript{\y}{}{\grad}^{t} = 
\begin{Bmatrix}
\prescript{\y}{}{\gradt}^{\!\!t}\\[3pt]
\frac{\partial}{\partial y_{N}}
\end{Bmatrix} =
\begin{bmatrix}
I & (1 - \epsilon^{-1}) \prescript{\x}{}{\gradt}^{t} \zeta\\[3pt]
\bm{0} & \epsilon^{-1}
\end{bmatrix}
\begin{Bmatrix}
\prescript{\x}{}\gradt^{t} \\[3pt]
\partial_{z}
\end{Bmatrix} .
\end{equation}
In the block matrix notation above, it is understood that $ I $ is the identity matrix in $ \R^{(N-1)\times(N-1)} $, $ \gradt \zeta, \bm{0} $ are vectors in $ \R^{N -1} $ and $ \partial_{z} = \frac{\partial}{\partial z} $.
%
In order to apply these changes to the gradient of a vector function $ \w $, we recall the matrix notation
\begin{equation}\label{Eq gradient structure}
\prescript{\y}{}\grad\,\w 
=
\begin{bmatrix}
\prescript{\y}{} \grad w_{\scriptscriptstyle 1}\\
\vdots \\
\prescript{\y}{}\grad w_{\scriptscriptstyle N}
\end{bmatrix}
%
= \begin{bmatrix}
\prescript{\x}{}\gradt w_{1} + (1 - \epsilon^{-1})\partial_{z} w_{1} \prescript{\x}{}\gradt \zeta
& \epsilon^{-1}\,\partial_{z} w_{1}\, \\
\vdots & \vdots \\
\prescript{\x}{}\gradt w_{N} + (1 - \epsilon^{-1})\partial_{z} w_{N}\prescript{\x}{}\gradt \zeta
& \epsilon^{-1}\, \partial_{z} w_{N}\,
\end{bmatrix} ,
\end{equation}
reordering we get
\begin{equation}\label{Eq reorganized gradient structure}
\prescript{\y}{}\grad\w \left(\xthilde,\,x_{\scriptscriptstyle N}\right)
= 
\begin{bmatrix}
\prescript{\x}{}\D^{\epsilon}\w &
\dfrac{1}{\epsilon}\, \partial_{z}\,\w
\end{bmatrix} .
\end{equation}
%
%
Here, the operator $ \prescript{\x}{}\deps $ is defined by
\begin{equation}\label{Def operator D epsilon}
\prescript{\x}{} \deps\w \defining 
\prescript{\x}{}\gradt \w 
+\Big(1-\frac{1}{\epsilon}\Big)\partial_{z}\w \prescript{\x}{}\gradt^{t}\zeta,
\end{equation}
i.e., $ \prescript{\x}{}\deps\,\w\in \R^{N\times(N-1)} $ and it is introduced to have a more efficient notation. In the next section we address the interface conditions. 
%
%
\subsubsection*{Divergence Operator} 
\noindent Observing the diagonal of the matrix in \eqref{Eq gradient structure} we have
\begin{equation}\label{Eq structure divergence canonical basis}
%
\prescript{\y}{}\div
\, \w\,(\xthilde,x_{\scriptscriptstyle N})
=\Big(\prescript{\x}{}\divt\,\w_{\scriptscriptstyle T}
+\Big(1- \frac{1}{\epsilon}\Big) \partial_{z}\,
\w_{\scriptscriptstyle T}\cdot \prescript{\x}{}\divt\,\zeta\Big)(\xthilde,z) 
+ \frac{1}{\epsilon}\,\partial_{z}\w_{\scriptscriptstyle N}\left(\xthilde,\,z\right) \\
%
%
\end{equation}
%
%
%
%
\begin{remark}\label{Rem Prescript}
The prescript indexes $ \y, \x $ written on the operators above were used only to derive the relation between them 
, however they will be dropped once the context is clear.
\end{remark}
%
%
\subsubsection*{Local vs Global Vector Basis}
\noindent It shall be seen later on, that the velocities of the channel need to be expressed in terms of an orthonormal basis $ \B $, such that the normal vector $ \n $ belongs to $ \B $ and the remaining vectors are locally tangent to the interface $ \Gamma $. Since $ \zeta: G \rightarrow \R $ is a $ C^{2} $ function it follows that $ \xthilde\mapsto \n\pxthilde $ is at least $ C^{1} $. 
\begin{definition}\label{Def Local System of Coordinates}
Let $ \B_{0} \defining  \big\{ \eversor_{1}, \ldots, \eversor_{N-1}, \eversor_{N} \big\} $ be the standard canonical basis in $ \R^{N} $.
For any $ \xthilde \in G $ let $ \B = \B(\xthilde)  \defining \big\{\nuversor_{1} , \ldots, \nuversor_{N - 1}, \n \big\} $ be an orthonormal basis in $ \R^{N} $. Define the linear map $ U(\xthilde) : \R^{N} \rightarrow \R^{N} $ by
\begin{align}
& U(\xthilde) \nuversor_{i} \defining \eversor_{i}, \text { for } i = 1, \ldots, N-1, &
&  U(\xthilde) \n \defining \eversor_{N}  .
\end{align}
We say the map $ \widetilde{\x} \mapsto U(\widetilde{\x}) $ is a \textbf{stream line localizer} if it is of class $ C^{1} $. In the sequel we write it with the following block matrix notation  
\begin{equation}\label{Eq blocks of the rotational matrix}
U \pxthilde \defining
\begin{bmatrix}
\UxTtang & \UxTnormal\\
\UxNtang & \UxNnormal
\end{bmatrix} .
\end{equation}
Here, the indexes $ T $ and $ N $ stands for the first $ N - 1 $ components and the last component of the vector field. The indexes $ \tang $ and $ \n $ indicate the tangent and normal directions to the interface $ \Gamma $.
\end{definition}
\begin{remark}\label{Rem Stream Line Localizer}
\begin{enumerate}[(i)]
\item Since $ \zeta \in C^{2}(G) $ and bounded, it is clear that for each $ \xthilde \in G $ basis $ \B  = \big\{\nuversor_{1}, \ldots, \nuversor_{N-1}, \n \big\} $ can be chosen, so that $ \xthilde \mapsto U(\xthilde) $ is $ C^{1} $. In the following it will be assumed that $ U $ is a stream line localizer.

\item Notice that by definition $ U(\xthilde) $ is an orthogonal matrix for all $ \xthilde \in G $. 
\end{enumerate}
\end{remark}
Next, we express the velocities in the local basis $ \wtwo $
in terms of the normal and tangential components, using the following relations
\begin{subequations}\label{Eq definition of normal and tangential velocity}
\begin{equation}\label{Eq definition of normal velocity}
\wnormalx = \wtwo\cdot\n \pxthilde ,
\end{equation}
\begin{equation}\label{Eq definition of tangential velocity}
\wtangx = 
\begin{Bmatrix}
\wtwo\cdot \nuversor_{1}(\xthilde) \\
\vdots\\
\wtwo\cdot\nuversor_{N - 1}(\xthilde )
\end{Bmatrix} .
\end{equation}
\end{subequations}
Clearly, the relationship between velocities is given by
\begin{equation}\label{Eq local and global velocities}
\wtwo \big(\xthilde, x_{\scriptscriptstyle N}\big)
= U \pxthilde 
\begin{Bmatrix}
\wtangx\\[3pt]
\wnormalx
\end{Bmatrix}
\big(\xthilde, x_{\scriptscriptstyle N}\big)\\
= \begin{bmatrix}
\UxTtang & \UxTnormal\\[3pt]
\UxNtang & \UxNnormal
\end{bmatrix}
\begin{Bmatrix}
\wtangx\\[3pt]
\wnormalx
\end{Bmatrix}
\big(\xthilde, x_{\scriptscriptstyle N}\big) .
\end{equation}
\begin{remark}
We stress the following observations
\begin{enumerate}[(i)]
\item The procedure above does not modify the dependence of the variables, only the way velocity field are expressed as linear combinations of a convenient (stream line) orthonormal basis.

\item The fact that $ U $ is a smooth function allows to claim that $ \wtang\in \big[H^1(\Omega_2)\big]^{N-1} $ and $ \wnormal\in H^1(\Omega_2) $. 

\item In order to keep notation as light as possible the dependence of the matrix $ U $ as well as the normal and tangential directions $ \n, \tang $ will be omitted whenever is not necessary to show it. 

\item Notice that given any two flow fields $ \utwo, \wtwo $ the following isometric identities hold
\begin{equation}\label{Eqn Isometric Equalities}
\begin{split}
\utwo_{\tang} \cdot \wtwo_{\tang}
& =
\utwo(\tang) \cdot \wtwo(\tang) ,\\
\utwo_{\n} \wtwo_{\n}
& =
\utwo(\n) \cdot \wtwo(\n)  ,
\\
\utwo \cdot \wtwo & = 
\utwo(\tang) \cdot \wtwo(\tang) + 
\utwo(\n) \cdot \wtwo(\n) =
\utwo_{\tang} \cdot \wtwo_{\tang} +
\utwo_{\n} \wtwo_{\n} .
\end{split}
\end{equation}
%

\end{enumerate}
\end{remark}
\begin{proposition}\label{Th Orthogonal Basis and Spaces Isometry}
Let $ \wtwo \in \mathbf{H}^{1}(\Omega_{2}) $ and let $ \wnormal, \wtang $ be as defined in \eqref{Eq definition of normal and tangential velocity}, then 
\begin{enumerate}[(i)]
\item 
\begin{equation}\label{Eq partial z velocities}
\partial_z \wtwo \big(\xthilde, x_{\scriptscriptstyle N}\big)
= U \pxthilde 
\begin{Bmatrix}
\partial_z\,\wtangx\\[3pt]
\partial_z\,\wnormalx
\end{Bmatrix}
\big(\xthilde, x_{\scriptscriptstyle N}\big) .
\end{equation}
\item 
\begin{equation}\label{Eq norms preserved}
\big\Vert \partial_z \wtwo \big\Vert ^{2}_{0,\Omega_2} =
\big\Vert \partial_z \wtan \big\Vert ^{2}_{0,\Omega_2}
+
\big\Vert \partial_z \wnorm \big\Vert ^{2}_{0,\Omega_2}\\
=
\big\Vert \partial_z \wtang \big\Vert ^{2}_{0,\Omega_2}
+
\big\Vert \partial_z \wnormal \big\Vert ^{2}_{0,\Omega_2} .
\end{equation}
\end{enumerate}
\end{proposition}
\begin{proof}
\begin{enumerate}[(i)]
\item It suffices to observe that the orthogonal matrix $ U $ defined in \eqref{Eq local and global velocities} is independent from $ z $.

\item Due to \eqref{Eq partial z velocities} we have
\begin{equation*}
\begin{split}
\big\vert \partial_z \wtwo \left(\xthilde, x_{\scriptscriptstyle N}\right)\big\vert ^{2} 
& =
\partial_z \wtwo \left(\xthilde, x_{\scriptscriptstyle N}\right)\cdot
\partial_z \wtwo \left(\xthilde, x_{\scriptscriptstyle N}\right) \\
& = U \pxthilde
\begin{Bmatrix}
\partial_z\,\wtangx\\
\partial_z\,\wnormalx
\end{Bmatrix}
\left(\xthilde, x_{\scriptscriptstyle N}\right)
\cdot
U \pxthilde 
\begin{Bmatrix}
\partial_z\,\wtangx\\
\partial_z\,\wnormalx
\end{Bmatrix}
\left(\xthilde, x_{\scriptscriptstyle N}\right) 
\\
& = \bigg\vert
\begin{Bmatrix}
\partial_z\,\wtangx\\
\partial_z\,\wnormalx
\end{Bmatrix}
\bigg\vert^{2} .
\end{split} 
\end{equation*}
The last equality holds true since the matrix $ U(\xthilde) $ is orthogonal at each point $ \xthilde \in G $, therefore it is an isometry in the Hilbert space $ \big(\R^{N}, \cdot \big) $. Recalling that $ \big\vert \partial_z \wtan \big(\xthilde, x_{\scriptscriptstyle N}\big)\big\vert ^{2}
+
\big\vert \partial_z \wnorm \big(\xthilde, x_{\scriptscriptstyle N}\big)\big\vert ^{2}
 =
\big\vert \partial_z \wtwo \big(\xthilde, x_{\scriptscriptstyle N}\big)\big\vert ^{2} $ for all $ \x = (\xthilde, x_{\scriptscriptstyle N}) $, the result follows.
%
\qed
\end{enumerate}
\end{proof}
%
%
%
\subsection{Interface Conditions and the Strong Form}
\noindent The interface conditions need to account for stress and mass balance. We start with the stress, to that end we decompose it in the tangential and normal components, the former is handled by the {\em Beavers-Joseph-Saffman}
\eqref{Beavers - Joseph -Saffman condition} condition and the latter by the classical {\em
Robin} boundary condition \eqref{Eq Interface Normal Stress Balance}, this gives
\begin{subequations}\label{Eq Interface Stress Balance}
\begin{equation}\label{Beavers - Joseph -Saffman condition}
\sigma^{2}(\n)_{\tang} =  
\epsilon^{2}\,\beta \, \sqrt{\Q}\, \vtwo(\tang) \,,
\end{equation}
\begin{equation}\label{Eq Interface Normal Stress Balance}
\sigma^{2}(\n)_{\n} - \ptwo + \pone 
= \alpha\, \vone \cdot \n   \text{ on } \Gamma .
\end{equation}
\end{subequations}
In the expression \eqref{Beavers - Joseph -Saffman condition} above,
$\epsilon^{2}$ is a scaling factor destined to balance out the
geometric singularity introduced by the thinness of the channel.  In
addition, the coefficient $\alpha \ge 0$ in \eqref{Eq Interface Normal
Stress Balance} is the {\em fluid entry resistance}. 
\newline
\newline
\noindent Next, recall that the stress satisfies $ \sigma^{2} = 2\,\epsilon\,\E\, \symgrad(\vtwo) $ (where the scale $ \epsilon $ is introduced according to the thinness of the channel) and that $\grad\cdot \vtwo = 0$ (since the system is conservative); then we have
\begin{equation*}\label{Eq stress divergence to laplacian velocity}
\grad\cdot\sigma^{2} = 
\grad\cdot \big[2\,\epsilon\,\E\,\symgrad\big(\vtwo\big) \big]
= \epsilon\,\E \grad\cdot\grad\vtwo .
\end{equation*}
Replacing in the equations \eqref{Eq Interface Stress Balance} we have the following set of interface conditions
\begin{subequations}\label{EQ interface conditions}
\begin{equation}\label{Beavers - Joseph condition}
\epsilon\,\mu\, 
\Big(\frac{\partial \,\vtwo}{\partial\, \n} - 
\big( \frac{\partial \, \vtwo}{\partial\, \n} \cdot \n\big)\n \Big)
= \epsilon^{2} \beta \,  \sqrt{\Q}\, \vtwo(\tang) \,,
\end{equation}
\begin{equation}  \label{scaled interface normal stress}
\epsilon\,\mu\Big(\frac{\partial \,\vtwo}{\partial\,\n}\cdot\n\Big)-p^{2}+p^{1}
=\alpha\,\vone\cdot\n \text{ on } \Gamma .
\end{equation}
\begin{equation}\label{admissability constraint}
\vone \cdot \n = \vtwo \cdot \n  \text{ on } \Gamma .
\end{equation}
\end{subequations}
The condition \eqref{admissability constraint} states the fluid flow (or mass) balance.
\newline
\newline
With the previous considerations, the Darcy-Stokes coupled system in terms of velocity $ \v $ and pressure $ p $ is given by 
\begin{subequations}\label{Eq Darcy-Stokes System} 
\begin{equation}\label{Eq Darcy Mass Conservation}
\div \vone = h_1 \,,
\end{equation}
\begin{equation}\label{darcy}
\Q \,\vone +  \grad \pone = \0 \,, \quad \text{in }  \Omega_{1} .
\end{equation}
\begin{equation}\label{Eq Divergence free condition}
\grad\cdot \vtwo = 0 ,
\end{equation}
\begin{equation}\label{Momentum fluid}
-\grad\cdot 2\epsilon \mu \symgrad (\vtwo)
+\grad \ptwo = \f_{2} .
\end{equation}
\end{subequations}
Here, equations \eqref{Eq Darcy Mass Conservation}, \eqref{darcy} correspond to the Darcy flow filtration through the porous medium, while equations \eqref{Eq Divergence free condition} and \eqref{Momentum fluid} stand for the Stokes free flow. Finally, we adopt the following boundary conditions
\begin{subequations}\label{BC Top boundary condition}
\begin{equation}\label{BC Drained}
\pone = 0 \, \quad \text{on }\, \partial \Omega_{1} - \Gamma .
\end{equation}
\begin{equation}\label{BC Null Flux}
\vtwo=0 \, \quad \text{on }\, 
\partial \Omega_{2}^{\epsilon} - \big(\Gamma + \epsilon \big) .
\end{equation}
\begin{equation}\label{BC tangential boundary condition}
\frac{\partial \,\vtwo}{\partial\,\n}
- \Big(\frac{\partial \,\vtwo}{\partial\,\n}\cdot\n\Big)\n
= 0\quad\text{on}\;\Gamma+\epsilon ,
\end{equation}
\begin{equation}\label{BC normal boundary condition}
\vtwo\cdot\n =
\vnormal = 0 \quad\text{on}\;\Gamma+\epsilon .
\end{equation}
\end{subequations}
The system of equations \eqref{Eq Darcy-Stokes System}, \eqref{BC Top boundary condition} and \eqref{EQ interface conditions} constitute the strong form of the Darcy-Stokes coupled system.
\begin{remark}\label{Rem Scaling Coments}
\begin{enumerate}[(i)]
\item For a detailed exposition on the system's adopted scaling namely, the
fluid stress tensor $ \sigma^{2} = 2\,\epsilon\,\E\, \symgrad(\vtwo) $ and the
Beavers-Joseph-Saffman condition \eqref{Beavers - Joseph -Saffman
condition}, together with the \textbf{formal} asymptotic analysis we refer to
\cite{Morales1}.

\item A deep discussion on the role of each physical variable and parameter in equations \eqref{Eq Darcy-Stokes System} as well as the meaning of the boundary conditions \eqref{BC Top boundary condition}, can be found in Sections 1.2, 1.3 and 1.4 in \cite{ShowMor17}.
\end{enumerate}
\end{remark}

\subsection{Weak Variational Formulation and a Reference Domain}\label{sec-wellposed}
%
\noindent In this section we present the weak variational formulation of the problem defined by the system of equations \eqref{Eq Darcy-Stokes System}, \eqref{BC Top boundary condition} and \eqref{EQ interface conditions}, on the domain $\Omega^\epsilon$, next, we rescale
$\Omega_2^\epsilon$ to get a uniform domain of reference. We begin defining the function spaces where the problem is modeled
\begin{definition}\label{Def Function Spaces}
Let $ \Omega, \Omega_{1}, \Omega_{2}^{\epsilon}, \Gamma $ be as introduced in Section \ref{Sec Geometric Setting}; in particular $ \Omega_{2} $ and $ \Gamma $ satisfy Hypothesis \ref{Hyp Interface Geometric Conditions}. Define the spaces
\begin{subequations}\label{Eq Function Spaces}
\begin{equation}\label{Def Velocities Two}
\X_{2}^{\epsilon}\defining
\big\{\v\in\H1bold(\Omega_{2}^{\epsilon}):\v=0 \text{ on }  
\partial \Omega_{2}^{\epsilon} - \big(\Gamma + \epsilon \big), \; 
\v\cdot\n = 0\;\text{on}\,\Gamma+\epsilon\, \big\} ,
\end{equation}
\begin{equation}\label{Def Velocities Match}
\X^{\epsilon} \defining \big\{[\,\v^{1},\v^{2}\,]\in
\Hdiv(\Omega_{1}^{\epsilon})\times\X_{2}^{\epsilon}:
\v^{1}\cdot\n=\v^{2}\cdot\n
\;\text{on}\,\Gamma\big\} 
=
\big\{ \v \in \Hdiv(\Omega^\epsilon): \v^{2} \in \X_2^\epsilon \big\},
\end{equation}
\begin{equation}\label{Def Pressures}
\Y^{\epsilon}\defining
L^{2}(\Omega^{\epsilon}) ,
\end{equation}
\end{subequations}
endowed with their respective natural norms. Moreover, for $ \epsilon  = 1 $ we simply write $ \X $, $ \X_{2} $ and $ \Y $.
\end{definition}
In order to attain well-posedness of the problem, the following hypothesis is adopted.
\begin{hypothesis}\label{Hyp Bounds on the Coefficients}
It will be assumed that $\mu > 0$ and the coefficients $\beta$ and
$\alpha$ are nonnegative and bounded almost everywhere. Moreover, the
tensor $\Q$ is elliptic, i.e., there exists a $C_{\Q}> 0$ such that $(\Q
\,\x)\cdot \x \geq C_{\Q} \Vert \x\Vert^{2}$ for all $\x\in \R^{N}$.
\end{hypothesis}
\begin{theorem}\label{Th Weak Variational Formulation}
Consider the boundary-value problem defined by the equations \eqref{Eq Darcy-Stokes System}, the interface coupling conditions \eqref{EQ interface conditions} and the boundary conditions \eqref{BC Top
boundary condition} then
\begin{enumerate}[(i)]
\item A weak variational formulation of the problem is given by
\begin{subequations}\label{scaled problem}
\begin{flushleft}
$\big[\,\v^{\epsilon},p^{\epsilon}\,\big]\in\X^{\epsilon}\times\Y^{\epsilon}: $
\end{flushleft}
\vspace{-0.4cm}
\begin{align}
\label{scaled problem 1}
\int_{\Omega_1} \big( \Q \,\v^{1,\,\epsilon} \cdot \w^1 & 
-
\pepsone\,\grad\cdot\wone
\big)\,d\y
+ \int_{\Omega_2^{\epsilon}} \big(\,\epsilon\,\E
\grad\,\vepstwo - \pepstwo
\delta\ \big)\bcol \grad\wtwo
 \,d\widetilde{\y} d\y_{\scriptscriptstyle N}
\\
+\,\alpha\int_{\Gamma}\big(\vepstwo\,\cdot\n\big)\, & \big(\wtwo\cdot\n\big)\,dS 
+ \int_{\Gamma} \epsilon^{2} \, \beta \, \sqrt{\Q}
\;\vtangeps \cdot \wtang \,d S 
= \int_{\Omega_2^{\epsilon}} {\mathbf f^{\,2,\,\epsilon}} \cdot
\wtwo \,d\widetilde{\y} \,d\y_{\scriptscriptstyle N} ,
\nonumber
\\
\label{scaled problem 2}
\int_{\Omega_1} \grad\cdot\vepsone\,
\varphi^1 \,d\y &
+ \int_{\Omega_2^{\epsilon}}  \grad\cdot\vepstwo\, \varphi^2
 \,d\widetilde{\y} \,d\y_{\scriptscriptstyle N}
= \int_{\Omega_1} h^{1,\,\epsilon} \, \varphi^1 \, d\y ,
\end{align}
\begin{flushright}
\text{for all } $\big[\w,\varphi\big]\in\X^{\epsilon}\times\Y^{\epsilon}$.
\end{flushright}
\end{subequations}
\item The problem \eqref{scaled problem} is well-posed.

\item The problem \eqref{scaled problem} is equivalent to 
\begin{subequations}\label{problem fixed geometry}
\begin{flushleft}
$ [\v^{\,\epsilon},\p^{\,\epsilon}]\in\X\times\Y : $
\end{flushleft}
\begin{multline}\label{problem fixed geometry 1}
\int_{\Omega_1}  \Q\, \vepsone \cdot \w^1 \,d\x -
\int_{\Omega_1}\pepsone \, \div \wone
\,d\x\\
-  \epsilon \int_{\Omega_2} \pepstwo \, \divt
\wtan \,d\widetilde{\x} \,dz
- \epsilon\Big(1-\frac{1}{\epsilon}\Big)\int_{\Omega_2} \pepstwo \, \partial_{z}\, \wtan\cdot\gradt \zeta \,d\widetilde{\x} \,dz
- \int_{\Omega_2} \pepstwo \, \partial_{z} \wnorm \,d\widetilde{\x} \,dz \\
+\epsilon^{2}\int_{\Omega_2} \E\,\deps\vepstwo \bcol \deps\wtwo\,d\widetilde{\x} \,dz 
+\int_{\Omega_2} \E\,\partial_{z}\,\vtaneps\cdot\partial_{z}\,\wtan \,d\widetilde{\x} \,dz
+\int_{\Omega_2} \E\,\partial_{z}\,\vnormeps\cdot\partial_{z}\,\wnorm \,d\widetilde{\x} \,dz\\
+\alpha
\int_{\Gamma}\big(\vepsone\cdot\n\big)\,\big(\wone\,\cdot\n\big)\,dS
+\epsilon^{2} \int_{\Gamma} \beta \sqrt{\Q} \,\vtangeps \cdot
\wtang\,d S 
= \epsilon\,\int_{\Omega_2} {\f^{2,\epsilon}} \cdot
\wtwo \,d\widetilde{\x} \,dz ,
\end{multline}
\begin{multline}\label{problem fixed geometry 2}
\int_{\Omega_1}\div \vepsone
\varphi^1 \,d\x  
+ \epsilon \int_{\Omega_2}  \divt\vtaneps
\,\varphi^2
 \,d\widetilde{\x} \,dz+
\epsilon\Big(1-\frac{1}{\epsilon}\Big)\int_{\Omega_2} \partial_{z} \vtaneps\cdot\gradt\zeta\,\varphi^{2} \,d\widetilde{\x} \,dz\\
+  \int_{\Omega_2}  \partial_{z} \vnormeps \,\varphi^2
\,d\widetilde{\x} \,dz
= \int_{\Omega_1} h^{1,\,\epsilon} \, \varphi^1 \, d\x ,
\end{multline}
\begin{flushright}
$ \text{for all } [\w,\Phi]\in\X\times\Y $.
\end{flushright}
\end{subequations}
\end{enumerate}
\end{theorem}

\begin{proof}
\begin{enumerate}[(i)]
\item See Proposition 3 in \cite{ShowMor17}, we simply highlight that the term $ \int_{\Omega_{2}} \epsilon^{2} \beta \sqrt{\Q} \, \vtwo(\tang) \cdot \wtwo(\tang) \, dS $ has been replaced by $ \int_{\Omega_{2}} \epsilon^{2} \beta \sqrt{\Q} \, \vtang \cdot \wtang \, dS $, due to the isometric identities \eqref{Eqn Isometric Equalities}. 

\item See Theorem 6 in \cite{ShowMor17}. The technique identifies the operators $ \A, \B, \C $ in the variational statements \eqref{scaled problem 1} and \eqref{scaled problem 2}, then it verifies that these operators satisfy the hypotheses of Theorem \ref{Th well posedeness mixed formulation classic}; the result delivers well-posedness.

\item A direct substitution of the expressions \eqref{Eq reorganized gradient structure} and \eqref{Eq structure divergence canonical basis} in the statements \eqref{scaled problem}, combined with the definition \eqref{Def operator D epsilon} yields the result. Also notice that the determinant of the matrix in the right hand side of the equation \eqref{Eq reorganized gradient structure} is equal to $ \epsilon^{-1} $. Finally, observe that the boundary conditions of space $ \X_{2}^{\epsilon}$, defined in \eqref{Def Velocities Two} are transformed into the boundary conditions of $ \X_{2} $ because none of them involve derivatives.
\qed
\end{enumerate}
\end{proof}
\begin{remark}\label{Rem Notation}
In order to prevent heavy notation, from now on, we denote the volume integrals by $ \int_{\Omega_{1}} F = \int_{\Omega_{1}} F \, d\x $ and $ \int_{\Omega_{2}} F = \int_{\Omega_{2}} F \, d\xthilde\,dz $. We will use the explicit notation $ \int_{\Omega_{2}} F \, d\xthilde\,dz $ only for those cases where specific calculations are needed. Both notations will be clear from the context.
\end{remark}

%
%
\section{Asymptotic Analysis}  \label{sec-convergence}
%
%
%
%
\noindent In this section, we present the asymptotic analysis of the problem i.e., we obtain a-priori estimates for the solutions $ \big((\veps, \peps): \epsilon > 0\big) $, derive weak limits and conclude features about those (velocity and pressure) limits. We start recalling a classical space. 
\begin{definition}\label{Def One Derivative Space}
Let $\Omega_{2} $ as in Definition \ref{Hyp Interface Geometric Conditions} and define the Hilbert space
\begin{subequations}\label{Def Higher Order Normal Trace Space}
\begin{align}\label{H partial z}
\Hpartial & \defining
\big\{ w\in L^{2}(\Omega_{2}):
\partial_{z}\,w\in\,L^{2}(\Omega_{2})\big\} ,\\
\Hboldpartial & \defining
\big\{ \w\in L^{2}(\Omega_{2}):
\partial_{z}\,\w\in\,L^{2}(\Omega_{2})\big\} ,
\end{align}
endowed with its natural inner product
\end{subequations}
\end{definition}
\begin{lemma}\label{Th Trace on One Derivative Space}
\begin{enumerate}[(i)]
\item 
Let $\Hpartial$ be the space introduced in Definition \ref{Def One
Derivative Space}, then the trace map $ w\mapsto w \big\vert_{\Gamma}$
from $\Hpartial$ to $L^{2}(\Gamma)$ is well-defined. Moreover, the
following Poincar\'e-type inequalities hold in this space
\begin{subequations}\label{Ineq Estimates on One Derivative Space}
\begin{equation}\label{Ineq Trace on One Derivative Space}
\Vert w \Vert_{0,\Gamma}\;
\leq 
\sqrt{2} \, \Big(\Vert w \Vert_{0, \Omega_{2}}
+ \Vert  \partial_{z}\,w\Vert_{0,\Omega_{2}} \Big),
\end{equation}
\begin{equation}\label{Ineq Conrol by Trace on One Derivative Space}
\Vert w \Vert_{0,\Omega_{2}}\;
\leq \sqrt{2} \, \Big(\Vert 
\partial_{z}\,w\Vert_{0,\Omega_{2}}
+\Vert w \Vert_{0,\Gamma} \Big),
\end{equation}
\end{subequations}
for all $ w \in \Hpartial $.

\item Let $ \Hboldpartial $ be the vector space introduced in Definition \ref{Def One Derivative Space} then, for any $ \w \in \Hboldpartial $ the estimates analogous to \eqref{Ineq Conrol by Trace on One Derivative Space} hold. 

\item Let $ \wtwo \in \mathbf{H}^{1}(\Omega_{2}) \subset \Hboldpartial $ and let $ \wnormal, \wtang $ be as defined in \eqref{Eq definition of normal and tangential velocity}, then
\begin{subequations}\label{Eq poincare new directions}
\begin{equation}\label{Eq poincare normal}
\big\Vert\, \wnormal\,\big\Vert_{0,\Omega_{2}}\;
\leq\big\Vert\,\partial_{z}\,\wnormal\,\big\Vert_{0,\Omega_{2}}
+2\,\Vert\,\wnormal\,\Vert_{0,\Gamma} ,
\end{equation}
\begin{equation}\label{Eq poincare tangential}
\big\Vert\, \wtang\,\big\Vert_{0,\Omega_{2}}\;
\leq\big\Vert\,\partial_{z}\,\wtang\,\big\Vert_{0,\Omega_{2}}
+2\,\Vert\,\wtang\,\Vert_{0,\Gamma} .
\end{equation}
\end{subequations}
\end{enumerate}
\end{lemma}
\begin{proof}
\begin{enumerate}[(i)]
\item 
The proof is a direct application of the fundamental theorem of
calculus on the smooth functions $ C^{\infty}(\Omega_{2}) $ which is a
dense subspace in $ \Hpartial $.

\item A direct application of equations \eqref{Ineq Estimates on One Derivative Space} on each coordinate of $ \w \in \Hboldpartial $ delivers the result.

\item It follows from a direct application of (i) and  (ii) on $ \wnormal $, $ \wtang $ respectively. 
\qed
\end{enumerate}
\end{proof}
Next we show that the sequence of solutions is globally bounded under the following hypotheses. 
\begin{hypothesis}\label{Hyp Bounds on the Forcing Terms}
In the following, it will be assumed that the sequences
$ (\f^{2,\epsilon}: \epsilon > 0)\subseteq
\mathbf{L}^{2}(\Omega_{2}) $ and $ (h^{1,\epsilon}: \epsilon >
0)\subseteq L^{2}(\Omega_{1}) $ are bounded, i.e., there exists $ C > 0 $
such that
\begin{align}\label{Ineq Bounds on the Forcing Terms}
& \big\Vert \f^{2, \epsilon} \big\Vert_{0, \Omega_{2} }
\leq  C,&
& \big\Vert h^{1, \epsilon} \big\Vert_{0, \Omega_{1} } 
\leq C,&
& \text{for all } \, \epsilon > 0 . 
\end{align}
\end{hypothesis}
\begin{theorem}[Global a-priori Estimate]\label{Th A-priori Estimates of Velocity}
Let $[\veps, \peps]\in \X\times \Y$ be the solution to the Problem
\eqref{problem fixed geometry}. There exists a constant $ K>0 $ such that
\begin{align}\label{general a priori estimate}
%
& \big\Vert\vepsone\big\Vert_{0,\Omega_{1}}^{2}
+\big\Vert 
\,\deps\big(\,\epsilon\,\vepstwo\,\big)\big\Vert_{0,\Omega_{2}}^{2}
 +\big\Vert \partial_{z}\vtaneps\big\Vert_{0,\Omega_{2}}^{2} 
%
+\big\Vert \partial_{z}\vnormeps\big\Vert_{0,\Omega_{2}}^{2}
+\big\Vert\vnormaleps \big\Vert_{0,\Gamma}^{2}
+\big\Vert\epsilon\,\vtangeps \big\Vert_{0,\Gamma}^{2}
 \leq K ,
&
&\text{for all }\, \epsilon > 0.
 %
\end{align}
\end{theorem}
\begin{proof}
Set $\w = \v^{\epsilon}$ in \eqref{problem fixed geometry 1} and
$\varphi = p^{\epsilon}$ in \eqref{problem fixed geometry 2} and add
them together. In addition, apply the
Cauchy-Bunyakowsky-Schwartz inequality to the right hand side and recall the
Hypothesis \ref{Hyp Bounds on the Coefficients}, this gives
\begin{multline}\label{a priori estimate}
\big\Vert \,\vepsone\, \big\Vert_{0, \Omega_{1}}^{2}
+\epsilon^{\,2}\int_{\Omega_{2}}
\deps\,\vepstwo:\deps\,\vepstwo 
%
+\big\Vert \, \partial_{z}\vtaneps\, \big\Vert_{0,\Omega_{2}}^{2}
+\big\Vert \, \partial_{z}\vnormeps\,\big\Vert_{0,\Omega_{2}}^{2}
+\,\big\Vert \,\vepsone\cdot\n \,\big\Vert_{0,\Gamma}^{2}
+\big\Vert \,\epsilon\,\vtangeps\,\big\Vert_{0,\Gamma}^{2}\\
\leq \frac{1}{k} \Big(\big\Vert\f^{2,\epsilon}_\tau\big\Vert_{0,\Omega_{2}} \big\Vert
\big(\epsilon\, \vtangeps\big) \big\Vert_{0,\Omega_{2}}
+\big\Vert\f^{2,\epsilon}_{\n}\big\Vert_{0,\Omega_{2}} \big\Vert
\big(\epsilon\, \vnormaleps \big)\big\Vert_{0,\Omega_{2}}
+\int_{\Omega_1} h^{1,\,\epsilon} \, \pepsone \, dx \,\Big) .
\end{multline}
The mixed terms were canceled out on the diagonal.  
We focus on the summand involving an integral in the right hand side of the expression above, i.e.,
\begin{equation}\label{estimate on the first integral}
\begin{split}
\int_{\Omega_1} h^{1,\epsilon} \, \pepsone \, dx
& \leq  \big\Vert \,
\pepsone \,\big\Vert_{0,\Omega_{1}}
\big\Vert \,h^{1,\epsilon} \,\big\Vert_{0,\Omega_{1}}\\
& \leq C 
\big\Vert \,\grad \pepsone \,\big\Vert_{0,\Omega_{1}}\big\Vert \,
h^{1,\epsilon} \,\big\Vert_{0,\Omega_{1}} 
=\big\Vert \, \Q\,\vepsone \,\big\Vert_{0,\Omega_{1}}
\big\Vert \,h^{1,\epsilon} \,\big\Vert_{0,\Omega_{1}}
\leq \widetilde{C}
\big\Vert \, \vepsone \,\big\Vert_{0,\Omega_{1}} .
\end{split} 
\end{equation}
The second inequality holds due to Poincar\'e's inequality, given that $ \pepsone = 0 $ on $ \partial \Omega_{1} - \Gamma $, as stated in Equation \eqref{BC Drained}. The equality holds due to \eqref{darcy}. The third inequality holds because the tensor $ \Q $ and the family of sources $ (h^{1,\epsilon}:\epsilon>0)\subset L^{2}(\Omega_{1}) $ are bounded as stated in Hypothesis \ref{Hyp Bounds on the Coefficients} and Hypothesis \ref{Hyp Bounds on the Forcing Terms} (Equation \eqref{Ineq Bounds on the Forcing Terms}), respectively. Next, we control the $ L^2(\Omega_{2}) $-norm of $ \vepstwo $. Since $ \vepstwo\in \mathbf{H}^{1}(\Omega_{2}) \subset \Hboldpartial $, the estimates \eqref{Eq poincare new directions} apply;
combining them with \eqref{estimate on the first
integral}, the bound \eqref{Ineq Bounds on the Forcing Terms} (from Hypothesis \ref{Hyp Bounds on the Forcing Terms}) in Inequality \eqref{a priori estimate} we have
\begin{equation*}
\begin{split}
\big\Vert \,\vepsone\, \big\Vert_{0,\Omega_{1}}^{\,2}
& 
+\epsilon^{2}\int_{\Omega_{2}}\,
\deps\,\vepstwo:\deps\,\vepstwo 
+\big\Vert \, \partial_{z}\vtaneps\, \big\Vert_{0,\Omega_{2}}^{\,2}
+\big\Vert \, \partial_{z}\vnormeps\,\big\Vert_{0,\Omega_{2}}^{\,2}
+\,\big\Vert \,\vepsone\cdot\n \,\big\Vert_{0,\Gamma}^{\,2}
+\big\Vert \,\epsilon\,\vtangeps \,\big\Vert_{0,\Gamma}^{\,2}\\
& \leq C\Big(\big\Vert\,
\partial_{z}\,\big(\epsilon\,\vtangeps\big)\,\big\Vert_{0,\Omega_{2}}
+2\,\big\Vert\,\big(\epsilon\,\vtangeps\big)\,\big\Vert_{0,\Gamma}
+\big\Vert\,
\partial_{z}\,\big(\epsilon\,\vnormaleps\big)\,\big\Vert_{0,\Omega_{2}}
+2\,\big\Vert\,\big(\epsilon\,\vnormaleps\big)\,\big\Vert_{0,\Gamma}
+\widetilde{C}
\big\Vert\,\vepsone\, \big\Vert_{0,\Omega_{1}} \,\Big)\\
& \leq  C \Big(\big\Vert\,
\partial_{z}\,\big(\epsilon\,\vtaneps\big)\,\big\Vert_{0,\Omega_{2}}
+2\,\big\Vert\,\big(\epsilon\,\vtangeps\big)\,\big\Vert_{0,\Gamma}
+\big\Vert\,
\partial_{z}\,\big(\epsilon\,\vnormeps\big)\,\big\Vert_{0,\Omega_{2}}
+2\,\big\Vert\,\big(\epsilon\,\vnormaleps\big)\,\big\Vert_{0, \Gamma}
+\widetilde{C}
\big\Vert
\,\vepsone\, \big\Vert_{0, \Omega_{1}} \Big) .
\end{split}
\end{equation*}
Using the equivalence of norms $\Vert \cdot \Vert_{\,1}\,,
\Vert \cdot \Vert_{\,2}$ for 5-D vectors yields
\begin{equation}\label{overall bound}
\begin{split}
\big\Vert \,\vepsone\, \big\Vert_{0,\Omega_{1}}^{\,2}
+\epsilon^{2} & \int_{\Omega_{2}}\,
\deps\,\vepstwo:\deps\,\vepstwo 
+\big\Vert \, \partial_{z}\vtaneps\, \big\Vert_{0,\Omega_{2}}^{\,2}
+\big\Vert \, \partial_{z}\vnormeps\,\big\Vert_{0,\Omega_{2}}^{\,2}
+\,\big\Vert \,\vepsone\cdot\n \,\big\Vert_{0,\Gamma}^{\,2}
+\big\Vert \,\epsilon\,\vtangeps \,\big\Vert_{0,\Gamma}^{\,2}\\
\leq C \Big\{ & \big\Vert\,
\partial_{z}\,\big(\epsilon\,\vtaneps\big)\,\big\Vert_{0,\Omega_{2}}^{2}
+\big\Vert\,\big(\epsilon\,\vtangeps\big)\,\big\Vert_{0,\Gamma}^{2}
+\big\Vert\,
\partial_{z}\,\big(\epsilon\,\vnormeps\big)\,\big\Vert_{0,\Omega_{2}}^{2}
+\big\Vert\,\big(\epsilon\,\vnormaleps\big)\,\big\Vert_{0,\Gamma}^{2}
+\widetilde{C}
\big\Vert
\,\vepsone\, \big\Vert_{0,\Omega_{1}}^{2} \,\Big\}^{1/2}\\
\leq C\Big\{ 
& 
\big\Vert \,\vepsone\, \big\Vert_{0,\Omega_{1}}^{2}
+\big\Vert\,\,\deps\big(\,\epsilon\,\vepstwo\,\big)\,\big\Vert_{0,\Omega_{2}}^{2}
+\big\Vert \, \partial_{z}\vtaneps\, \big\Vert_{0,\Omega_{2}}^{2} 
%
+\big\Vert \, \partial_{z}\vnormeps\,\big\Vert_{0,\Omega_{2}}^{2}
+\,\big\Vert \,\vnormaleps \,\big\Vert_{0,\Gamma}^{2}
+\big\Vert \,\epsilon\,\vtangeps \,\big\Vert_{0,\Gamma}^{2}\Big\}^{1/2} .
\end{split}
\end{equation}
From the expression above, the global Estimate \eqref{general a priori estimate} follows.
\qed
\end{proof}
In the next subsections we use weak convergence arguments to derive the functional setting of the limiting problem, see Figure \ref{Fig Sketch Limit} for the structure of the limiting functions.
\begin{corollary}[Convergence of the Velocities]\label{Th Direct Weak Convergence of Velocities}
Let $[\veps, \peps]\in \X\times \Y$ be the solution to the Problem
\eqref{problem fixed geometry}. There exists a subsequence, still
denoted $(\veps:\epsilon>0)$ for which the following holds.
\begin{enumerate}[(i)]
\item There exist $\vone\in \Hdiv(\Omega_{1}) $ such that
%
%
\begin{subequations}\label{Eq limit velocity one}
\begin{equation}\label{velocity omega 1}
\vepsone\rightarrow \vone\quad\text{weakly in}\quad
\Hdiv(\Omega_1) ,
\end{equation}
\begin{equation}\label{Eq velocity one conservation}
\div \vepsone  = h^{1}.
\end{equation}
\end{subequations}
%
%

\item There exist $ \boldsymbol{\chi}\in \mathbf{L}^{2}(\Omega_{2}) $ and $ \vtwo\in
\mathbf{H}^{1}(\Omega_{2}) $ such that 
\begin{subequations}\label{Stmt Higher Order Velocity Weak Convergence}
\begin{align}\label{Eq epsilon partial tangential L2 weak}
& \partial_{z}\vepstwo \rightarrow \boldsymbol{\chi} \quad\text{weakly in}\quad 
\mathbf{L}^{2}(\Omega_{2})\,, &
& \partial_{z} \big(\epsilon\,\vepstwo\big)
\rightarrow \bm{0} \quad \text{strongly in}\quad 
\mathbf{L}^{2}(\Omega_{2}),
\end{align}
\begin{align}\label{funcL2weak}
& \epsilon\,\vepstwo\rightarrow\vtwo
\quad\text{weakly in}\quad \mathbf{H}^{1}(\Omega_{2}) ,&
&\text{strongly in}\quad \mathbf{L}^{2}(\Omega_{2}) ,
\end{align}
moreover $ \vtwo $ satisfies 
\begin{equation}\label{solutionconvergence}
\vtwo =\vtwo\,(\xthilde) \, .
\end{equation}
\end{subequations}

\item There exists $ \xi\in \Hpartial $ such that
\begin{subequations}\label{Eq Limit Normal Velocity}
\begin{align}\label{convergence of normal velocity}
& \vnormaleps \rightarrow \xi \quad\text{weakly in }\, \Hpartial , &
& \big(\epsilon\,\vnormaleps\big)
\rightarrow 0 \quad \text{strongly in} \; \Hpartial ,
\end{align}
furthermore, $ \xi $ satisfies the interface and boundary conditions
\begin{align}\label{boundary conditions on normal velocity}
& \xi\big\vert_{\Gamma} = \vone\cdot \n\big\vert_{\Gamma} \, , &
& \xi\big\vert_{\Gamma + 1} = 0 .
\end{align}
\end{subequations}
\item The following properties hold
\begin{align}\label{Eq vtwo tangential and higher order terms relationship}
& \vtwo \cdot \n = 0 , &
& \boldsymbol{\chi}\cdot\n = \partial_{z}\xi.
\end{align}
\end{enumerate}
\end{corollary}
\begin{proof}
\begin{enumerate}[(i)]
\item (The proof is identical to part (i) Corollary 11 in \cite{ShowMor17}, we write it here for the sake of completeness.) Due to the global a-priori Estimate \eqref{general a priori
estimate} there must exist a weakly convergent subsequence $\vone\in \Hdiv(\Omega_{1}) $
such that \eqref{velocity omega 1} holds only in the weak
$\mathbf{L}^{2}(\Omega_{1})$-sense. Because of the hypothesis \ref{Hyp
Bounds on the Forcing Terms} and the equation \eqref{Eq Divergence free condition},
the sequence $ (\div\vepsone:\epsilon>0) \subset L^{2}(\Omega_{1})$
is bounded. Then, there must exist yet another subsequence, still
denoted the same, such that \eqref{velocity omega 1} holds in the weak
$\Hdiv(\Omega_{1})$-sense. Now recalling that the divergence operator is linear and continuous with respect to the $ \Hdiv $-norm the identity \eqref{Eq velocity one conservation} follows.
\item From the estimate \eqref{general a
priori estimate} it follows that $ (\partial_{z} \vtaneps: \epsilon> 0) $ and $ (\partial_{z} \vnormeps: \epsilon> 0) $ are bounded in $ \mathbf{L}^{2}(\Omega_{2}) $. Then, there exists a subsequence (still denoted the same) and $ \boldsymbol{\chi} \in \mathbf{L}^{2}(\Omega_{2}) $ such that $ (\partial_{z}\vepstwo: \epsilon > 0 ) $ and $ (\partial_{z}(\epsilon\, \vepstwo): \epsilon > 0 ) $ satisfy the statement \eqref{Eq epsilon partial tangential L2 weak}.
Also from \eqref{general a priori estimate} the trace on the interface $ \big(\epsilon\,\vepstwo\big\vert_{\,\Gamma}:\epsilon>0\big) $ is bounded in $ \mathbf{L}^{2}(\Gamma) $. Applying the inequality \eqref{Ineq Conrol by Trace on One Derivative Space} for vector functions, we conclude that $ \big(\epsilon\,\vepstwo:\epsilon>0\big) $ is bounded in $ \mathbf{L}^{2}(\Omega_{2}) $ and consequently in $ \Hboldpartial $; then there must exist $ \vtwo\in \Hboldpartial $ such that
\begin{align}\label{Eq convergence of velocity two in H partial}
\epsilon\,\vepstwo\rightarrow \vtwo\quad\text{weakly}\quad\text{in}\quad
\Hboldpartial .
\end{align}
Also, from the strong convergence in the statement \eqref{Eq epsilon partial tangential L2 weak}, it follows that $ \vtwo $ is independent from $ z $ i.e., \eqref{solutionconvergence} holds. 
\newline
\newline
Again, from \eqref{general a priori estimate} we know that the sequence $ \big(\epsilon\,\deps\,\vepstwo:\epsilon>0\big) $ is bounded in $ \mathbf{L}^{2}(\Omega_{2}) $, recalling the identity \eqref{Def operator D epsilon} we have that the expression 
\begin{equation*}
\epsilon\,\deps\,\vepstwo = \grad_{T}\left(\epsilon\,\vepstwo\right) + (\epsilon - 1)\partial_{z}\,\vepstwo\grad_{T}^{t}\,\zeta ,
\end{equation*}
is bounded. In the equation above the left hand side and the second summand of the right hand side are bounded in $ \mathbf{L}^{2}(\Omega_{2}) $, then we conclude that the first summand of the right hand side is also bounded. Hence, $ \big(\epsilon\,\grad \vepstwo:\epsilon>0\big) $ is bounded in $ \mathbf{L}^{2}(\Omega_{2}) $ and therefore the sequence $ \big(\epsilon\,\vepstwo:\epsilon>0\big) $ is bounded in $ \mathbf{H}^{1}(\Omega_{2}) $ consequently, the statement \eqref{funcL2weak} holds.

\item Since $ \big(\partial_z\,\vepstwo:\epsilon>0\big)\subset \mathbf{L}^2\left(\Omega_{2}\right) $ is bounded in particular $ \big(\partial_z\,\vepstwo\cdot\n:\epsilon>0\big)\subset L^2\left(\Omega_{2}\right) $ is also bounded. From \eqref{general a priori estimate}, we know that 
$ \big(\vepstwo\cdot\n\,\big\vert_{\Gamma}:\epsilon>0\big)\subset L^2\left(\Gamma\right) $ is bounded and again, due to inequality \eqref{Ineq Conrol by Trace on One Derivative Space} we conclude that $ \big(\vepstwo\cdot\n:\epsilon>0\big)\subset L^2\big(\Omega_{2}\big)$ is bounded. Then, the sequence $ \big(\vepstwo\cdot\n:\epsilon>0\big) $ is bounded in $ \Hpartial $ and there must exist $ \xi\in \Hpartial $ and a subsequence, still denoted the same, such that $ (\vnormaleps: \epsilon > 0) $ and $ (\epsilon \, \vnormaleps: \epsilon > 0) $ satisfy the statement \eqref{convergence of normal velocity}. From here it is immediate to conclude the relations \eqref{boundary conditions on normal velocity}.

\item Since $ (\epsilon \, \vepstwo\cdot\n)\rightarrow 0 $ and due to \eqref{Eq convergence of velocity two in H partial}, we conclude that $ \vtwo\cdot\n = 0 $. Finally, due to \eqref{Stmt Higher Order Velocity Weak Convergence}, we have that $ \boldsymbol{\chi}\cdot\n = \partial_{z}\,\xi $ and the proof is complete.
\qed
\end{enumerate}
\end{proof}%
\begin{figure} 
	%
	\begin{subfigure}[Limit Solutions in the Reference Domain]
			{ \includegraphics[scale = 0.36]{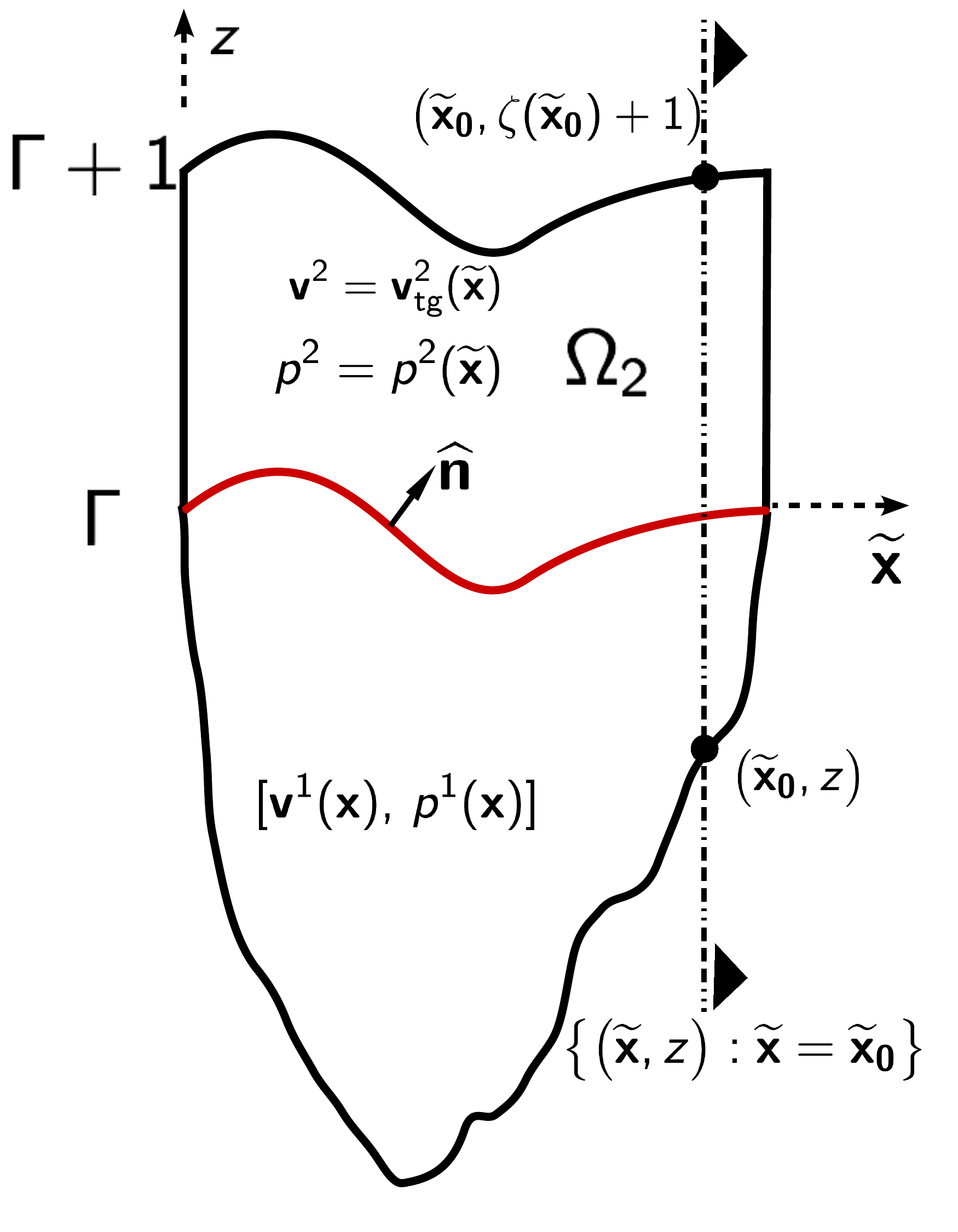} }
	\end{subfigure} 
	~ 
	\begin{subfigure}[Velocity and Pressure Schematic Traces for the solution on the hyperplane
		$ \{(\xthilde, z): \xthilde = \xthilde_{0}\} $.]
			{ \includegraphics[scale = 0.36]{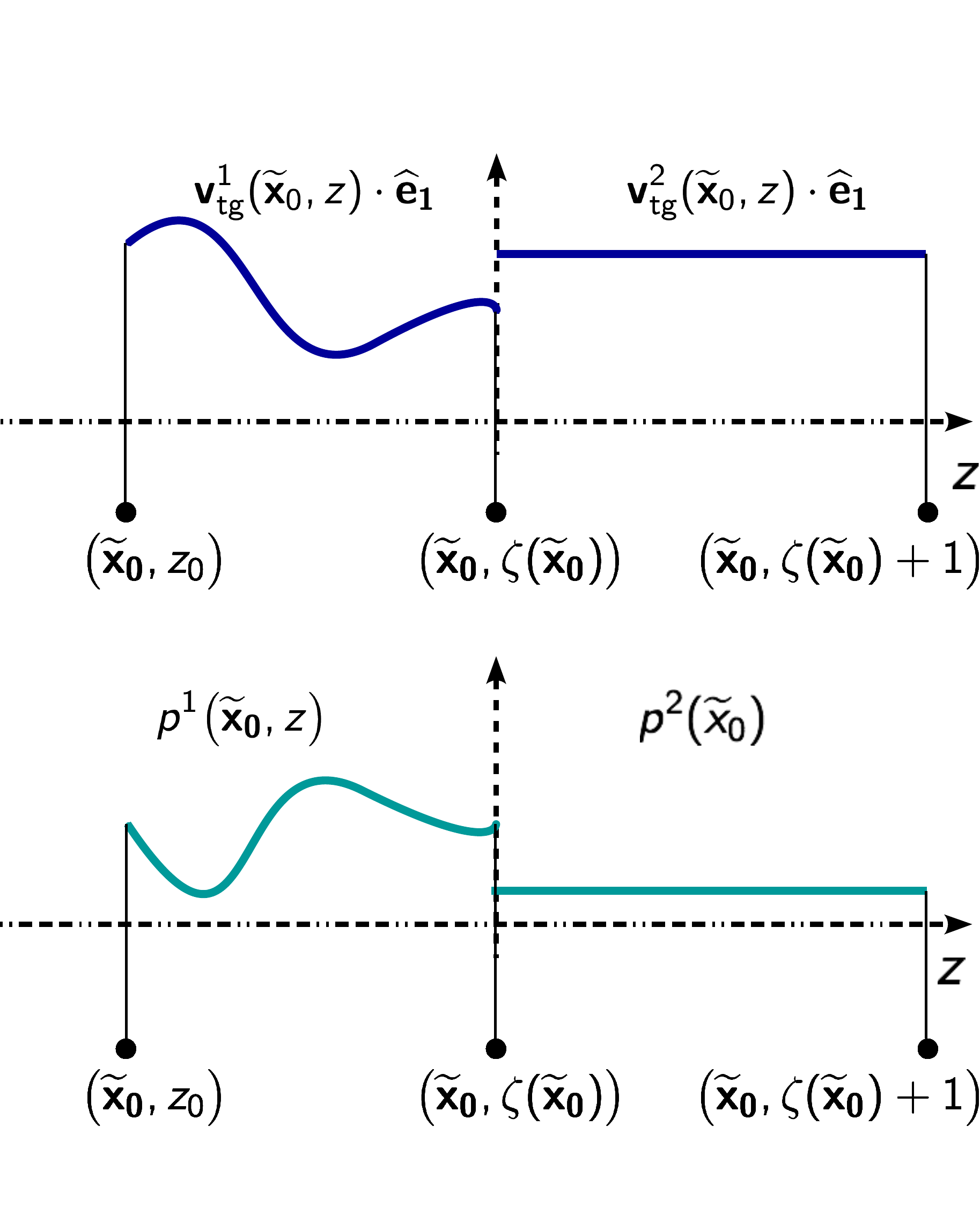} }             
	\end{subfigure} 
	\caption{Figure (a) depicts the dependence of the limit solution $ [\v, p] $, for both regions $ \Omega_{1} $ and $ \Omega_{2} $ and a generic hyperplane $ \{(\xthilde, z): \xthilde = \xthilde_{0}\} $. Figure (b) shows plausible schematics for traces of the velocity and pressure on the hyperplane $ \{(\xthilde, z): \xthilde = \xthilde_{0}\} $, depicted in Figure (a).}
	\label{Fig Sketch Limit}
\end{figure}
\begin{theorem}[Convergence of Pressures]\label{Th Convergence of Pressure One}
Let $ \big[\veps, \peps \big]\in \X\times \Y $ be the solution of problem \eqref{problem fixed geometry}. There exists a subsequence, still
denoted $ (\peps:\epsilon>0) $ verifying the following.
\begin{enumerate}[(i)]
\item There exists $\pone\in H^{1}(\Omega_{1})$ such that
\begin{subequations}
\begin{equation}\label{convergence of the pressure in Omega_1}
\pepsone \rightarrow \pone \text{ weakly in }
H^{1}(\Omega_{1}) \text{ and strongly in }\,
L^{2}(\Omega_{1}), 
\end{equation}
\begin{align}\label{Eq Equations and BC of pone}
& \Q\vone + \grad \pone = 0 \text{ in } \Omega_{1}, &
& \pone = 0 \text { on } \partial\Omega_{1} - \Gamma,
\end{align}
\end{subequations}
with $ \vone $ the weak limit of Statement \eqref{velocity omega 1}.
\item There exists $\ptwo\in L^{2}(\Omega_{2})$ such that
\begin{equation}\label{convergence of the pressure in Omega_2}
\pepstwo \rightarrow \ptwo\quad \text{weakly in }\,  L^{2}(\Omega_{2}) .
\end{equation}
\item The pressure $p = \big[ \pone, \ptwo \big]$ belongs to $L^{2}(\Omega)$.
\end{enumerate}
\end{theorem}
\begin{proof}
\begin{enumerate}[(i)]
\item (The proof is identical to part (i) Lemma 12 in \cite{ShowMor17}, we write it here for the sake of completeness.) Due to \eqref{darcy} and \eqref{a priori
estimate} it follows that
\begin{equation*}
\big\Vert \grad \pepsone\big\Vert_{0,\Omega_{1}}=
\big\Vert \sqrt{\Q}\,\vepsone\big\Vert_{0,\Omega_{1}}\leq C ,
\end{equation*}
with $ C > 0 $ an adequate positive constant. From \eqref{BC Drained}, the Poincar\'e inequality implies there exists
a constant $ \widetilde{C}>0 $ satisfying
\begin{align}\label{Ineq H^1 boundedness of pressure one}
& \big\Vert \pepsone\big\Vert_{1,\Omega_{1}}
\leq \widetilde{C} \, \big\Vert \grad\pepsone\big\Vert_{0,\Omega_{1}}, &
& \text{for all }\, \epsilon > 0  .
\end{align}
Therefore, the sequence $ (\pepsone: \epsilon > 0) $ is bounded in
$H^{1}(\Omega_{1})$ and the convergence statement \eqref{convergence of the
pressure in Omega_1} follows directly. Again, given that $ \pepsone $ satisfies the Darcy equation \eqref{darcy} and that the gradient $ \grad  $ is linear and continuous in $ H^{1} (\Omega_{1}) $ the equality $ \Q\vone + \grad \pone = 0 $ in \eqref{Eq Equations and BC of pone} follows. Finally, since $ \pepsone \big\vert_{\Omega_{1} - \Gamma} = 0 $ for every element of the weakly convergent subsequence and the trace map $ \varphi \mapsto \varphi\big\vert_{\Gamma} $ is linear, it follows that $ \pone $ satisfies the boundary condition in \eqref{Eq Equations and BC of pone}.  

\item 
In order to show that the sequence $ (\pepstwo: \epsilon > 0 ) $ is
bounded in $ L^{2}(\Omega_{2}) $, take any $ \phi\in
C_{0}^{\infty}(\Omega_{2}) $ and define the auxiliary function
\begin{align}\label{negative antiderivative}
& \varpi(\xthilde,z)\defining
\int_{z}^{\zeta(\xthilde) + 1}\phi(\xthilde,t)\,dt , &
& \zeta(\xthilde) \leq z \leq \zeta(\xthilde) + 1.
\end{align}
Since $ \zeta\in C^{2}(G) $, it is clear that $ \varpi \in H^{1}(\Omega_{2}) $ and $ \Vert \varpi \Vert_{\,1,\Omega_2}\leq C \Vert \phi \Vert_{0,\Omega_2} $. Hence, the function $ \w^{2}\defining \big(\0_{T},\varpi\big) = \varpi \, \eversor_{N} $ belongs to $ \X_{2} $, moreover
\begin{equation}\label{Eq estimate on the test function 1}
\begin{split}
\Vert\,\wtwo\,\Vert_{0,\Omega_2} +
\Vert\,\partial_{z}\,\wtwo\,\Vert_{0,\Omega_2}
& \leq \widetilde{C} \,  \Vert\,\phi\,\Vert_{0,\Omega_2}, \\
\Vert\wtwo(\tang)\big\vert_{\Gamma}\Vert_{0,\Gamma} =
\Vert\wtang\big\vert_{\Gamma}\Vert_{0,\Gamma}
& \leq \widetilde{C} \,  \Vert\,\phi\,\Vert_{0,\Omega_2} .
\end{split}
\end{equation}
Here, the second inequality follows from the first one and due to the estimate \eqref{Ineq Trace on One Derivative Space}. Next, observe that
$ \varpi \, \eversor_{N}\cdot \n(\cdot)\, \ind_{\Gamma} \in L^{2}(\Gamma)\subseteq
H^{-1/2}(\partial \Omega_{1}) $ then, Lemma \ref{Th Surjectiveness
from Hdiv to H^1/2} gives the existence of a function $ \wone\in
\Hdiv\,(\Omega_{1}) $ such that 
\begin{equation}\label{Eq Normal Trace Match}
\begin{split}
\wone\cdot\n 
=\w^{2}\cdot\n 
=\varpi(\xthilde,\zeta(\xthilde)) \eversor_{N} \cdot \n(\xthilde)
& = \int_{\zeta(\xthilde)}^{\zeta(\xthilde) + 1}\phi\,(\xthilde,t)\,dt  
\text{ on } \Gamma , \\
\wone \cdot \n & = 0  \text{ on } \partial \Omega_{1} - \Gamma, \\
\Vert \wone\Vert_{\,\Hdiv(\Omega_1)}
& \leq \Vert \varpi \, \eversor_{N} \cdot \n (\cdot) \Vert_{0,\Gamma}\leq C
\Vert \phi \Vert_{0,\Omega_2}.
\end{split}
\end{equation}
Here, the last inequality holds because $ \sup\{\n(\xthilde)\cdot \eversor_{N}: \xthilde \in \Gamma \} < \infty $.   
Hence, the function $ \w \defining [\w^{1},\w^{2}] $ belongs to the space $ \X $. Testing \eqref{problem fixed geometry 1} with $ \w $ yields
\begin{multline}\label{relation partial z pressure and normal velocity}
\begin{split}
\int_{\Omega_1} \Q\,\vepsone\cdot\wone 
& - \int_{\Omega_1} \pepsone\,
\div\wone 
+\int_{\Omega_2}\,
\pepstwo\,\phi 
+\epsilon^{2}\int_{\Omega_2}\E\,\deps\,\vepstwo
\bcol\,\deps\,\wtwo 
-\int_{\Omega_2}\E\,\partial_z\,\vnormeps\,\phi 
\\
& +\alpha\int_{\,\Gamma}\big(\vepsone\cdot\n\big)\big(\wone\cdot\n\,\big) d S
+\epsilon^{2} \int_{\,\Gamma}\gamma\,\sqrt{\Q}\,
\vtangeps\cdot\wtang d S
%
=\epsilon\int_{\Omega_2}\f^{2,\epsilon}_{N}\,\varpi .
\end{split}
\end{multline}
Applying the Cauchy-Bunyakowsky-Schwarz inequality to the integrals and reordering we get
\begin{equation*}
\begin{split}
\Big\vert\int_{\Omega_2}\,
\pepstwo\,\phi 
\Big\vert
\leq \, 
& C_1\,\big\Vert\,\vepsone\,\big\Vert_{0, \Omega_1}
\big\Vert\,\wone\,\big\Vert_{0, \Omega_1}
+\big\Vert\,\pepsone\,\big\Vert_{0, \Omega_1}
\big\Vert\,\div\wone\,\big\Vert_{0, \Omega_1}+
C_2\,\big\Vert\,\epsilon\,\deps\,\vepstwo\big\Vert_{0,\Omega_2}
\big\Vert\,\epsilon\,\deps\,\wtwo\,\big\Vert_{0,\Omega_2}
\\
& 
+
C_3\,\big\Vert\,\partial_z\,\vnormeps\,\big\Vert_{0,\Omega_2}
\big\Vert\,\phi\,\big\Vert_{0,\Omega_2}
+C_4\,\big\Vert\,\vepsone\cdot\n\,\big\Vert_{0,\Gamma} \big\Vert\,\wtwo\cdot\n\,\big\Vert_{0,\Gamma}
\\
& 
+C_5\,\epsilon\, \big\Vert\epsilon\,\vtangeps\big\Vert_{0,\Gamma} \big\Vert\wtang\big\Vert_{0,\Gamma}
+ \big\Vert\,\epsilon\,\f^{2,\epsilon}_{N}\,\big\Vert_{0,\Omega_2}
\big\Vert\,\varpi\,\big\Vert_{0,\Omega_2} .
\end{split}
\end{equation*}
We pursue estimates in terms of $ \Vert \phi \Vert_{0,\Omega_{2}} $, to that end we first apply the fact that all the terms involving the solution on the right hand side, i.e. $ \vepsone, \,\pepsone, \,\epsilon\,\deps\,\vepstwo,\, \partial_{z}\,\vnormeps, \, \vepsone\cdot\n $ and 
$ \vtangeps\big\vert_{\Gamma} $ are bounded; in addition, the forcing term $ \epsilon\,\f^{2,\epsilon}_{N} $ is bounded. Replacing these by a generic constant on the right hand side we have
\begin{equation}\label{Eq estimate on pressure two}
\begin{split}
\Big\vert\int_{\Omega_2}\,
\pepstwo\,\phi 
\Big\vert
\leq
& C\Big(\big\Vert\,\wone\,\big\Vert_{0, \Omega_1}
+\big\Vert\,\div\wone\,\big\Vert_{0, \Omega_1}
\\
& 
+
\big\Vert\,\epsilon\,\deps\,\wtwo\,\big\Vert_{0,\Omega_2}+
\big\Vert\,\phi\,\big\Vert_{0,\Omega_2} 
+ \big\Vert\,\wtwo\cdot\n\,\big\Vert_{0,\Gamma}
+ \epsilon\,\big\Vert\wtang\big\Vert_{0,\Gamma}
+ \big\Vert\,\varpi\,\big\Vert_{0,\Omega_2}\Big) .
\end{split}
\end{equation}
In the expression above the first summand of the second line needs further analysis, we have
\begin{equation*}
\begin{split}
\big\Vert\,\epsilon\,\deps\,\wtwo\,\big\Vert_{0,\Omega_2} 
& =
\big\Vert\,\epsilon\,\gradt\,\wtwo
+(\epsilon-1)\,\partial_{z}\,\wtwo\,\gradt^{t}\zeta\,\big\Vert_{0,\Omega_2}\\
& \leq \epsilon\,\big\Vert\,\gradt\,\wtwo
+\partial_{z}\,\wtwo\,\gradt^{t}\zeta\,\big\Vert_{0,\Omega_2}+
\big\Vert\,\partial_{z}\,\wtwo\,\gradt^{t}\zeta\,\big\Vert_{0,\Omega_2}\\
& \leq \epsilon\,\big\Vert\,\gradt\,\wtwo
+\partial_{z}\,\wtwo\,\gradt^{t}\zeta\,\big\Vert_{0,\Omega_2}+
\big\Vert\,\partial_{z}\,\wtwo\big\Vert_{0,\Omega_2}
\big\Vert\gradt^{t}\,\zeta\,\big\Vert_{0,\Omega_2} .
\end{split}
\end{equation*}
Combining \eqref{Eq estimate on the test function 1} with the expression above we conclude
\begin{equation}\label{Eq estimate on twisted tangential gradient}
\big\Vert\,\epsilon\,\deps\,\wtwo\,\big\Vert_{0,\Omega_2} \\
\leq \epsilon\,\big\Vert\,\gradt\,\wtwo
+\partial_{z}\,\wtwo\gradt^{t}\,\zeta\,\big\Vert_{0,\Omega_2}+
C\,\big\Vert \phi\big\Vert_{0,\Omega_2} .
\end{equation}
Back to the equation \eqref{Eq estimate on pressure two}, the two summands on the left hand side of the first line are bounded by a multiple of $\Vert\,\phi\,\Vert_{0,\Omega_{2}}$ due to \eqref{Eq estimate on the test function 1}. The first two summands on the third line are trace terms which are also controlled by a multiple of $\Vert\,\phi\,\Vert_{0,\Omega_{2}}$ due to \eqref{Eq estimate on the test function 1}. The third summand on the third line is trivially controlled by $\Vert\,\phi\,\Vert_{0,\Omega_{2}}$ due to its construction. Combining all these observations with \eqref{Eq estimate on twisted tangential gradient} we get
\begin{equation*}
\Big\vert\int_{\Omega_2}\,
\pepstwo\,\phi 
\Big\vert\\
\leq
\epsilon\,\big\Vert\,\gradt\,\wtwo
+\partial_{z}\,\wtwo \, \gradt^{t}\zeta\,\big\Vert_{0,\Omega_2}+
C\;\big\Vert \, \phi \, \big\Vert_{0,\Omega_2} , 
\end{equation*}
where $ C > 0 $ is a new generic constant. 
The last inequality holds since all the summands in the parenthesis
are bounded due to the estimates \eqref{general a priori estimate},
\eqref{Ineq H^1 boundedness of pressure one} and the Hypothesis
\ref{Hyp Bounds on the Forcing Terms}. Taking upper limit when
$\epsilon\rightarrow 0$, in the previous expression we get
\begin{equation}\label{bound on epsilon pressure in
L2}
\limsup_{\,\epsilon\,\downarrow\,0}\Big\vert\int_{\Omega_2}\,
\pepstwo\,\phi 
\Big\vert\leq C \, \big\Vert \phi\big\Vert_{0,\Omega_2} .
\end{equation}
The above holds for any $ \phi\in C_0^{\infty}(\Omega_{2}) $, then the sequence $ (\pepstwo:\epsilon>0\,)\subset
L^{2}(\Omega_{2})$ is bounded and consequently the convergence statement \eqref{convergence of
the pressure in Omega_2} follows.

\item From the previous part it is clear that the sequence
$ \big([\pepsone, \pepstwo]: \epsilon > 0\big) $ is bounded in
$ L^{2}(\Omega) $, therefore $ p $ also belongs to $ L^{2}\Omega) $ and the proof is complete.
\qed
\end{enumerate}
\end{proof}
\begin{remark}\label{Rem relations of the normal and derivative with respect to the normal}
Notice that the upwards normal vector $ \n $ orthogonal to the surface $ \Gamma $ is given by the expression
\begin{subequations}\label{Eq relations of the normal and derivative with respect to the normal}
\begin{equation}\label{Eq relation gradient surface and normal surface}
\n = \frac{1}{\vert(-\gradt \zeta, 1)\vert} 
\begin{Bmatrix}
-\gradt^{t}\zeta \\
1
\end{Bmatrix}  
\end{equation}
and the normal derivative satisfies
\begin{align}\label{Eq Derivative with respect to the normal}
\vert(-\gradt \zeta, 1)\vert \partial_{\,\n}  = 
\vert(-\gradt \zeta, 1)\vert \frac{\partial }{\partial \n }
= \n \cdot \grad  ,
\quad\text{ on } \Gamma .
\end{align}
%
\end{subequations}
\end{remark}
We use the identities above to identify the dependence of $ \boldsymbol{\chi} $, $ \xi $ and $ \ptwo $, see Figure \ref{Fig Sketch Limit}.
\begin{theorem}\label{Th Dependence of xi and Pressure 2}
Let $ \boldsymbol{\chi} $, $ \xi $ be the higher order limiting terms in Corollary
\ref{Th Direct Weak Convergence of Velocities} (ii) and (iii). Let $ \ptwo $ be the limit
pressure in $\Omega_{2}$ in Lemma \ref{Th Convergence of Pressure One}
(ii). Then 
\begin{subequations}\label{Eq Dependence of xi chi and Pressure 2}
%
\begin{align}  \label{partial xi dependence}
& \partial_z\,\xi = -\vone \cdot \n \big\vert_{\Gamma} ,&
& \xi (\xthilde, z) =  \vone\cdot \n (\xthilde)\big(1 - z  \big) ,
\quad \text{ for } \zeta(\xthilde) \leq z \leq \zeta(\xthilde) + 1 .
\end{align}
%
%
%
\begin{equation}  \label{Eq pressure two dependence}
\ptwo = \ptwo (\xthilde) .
\end{equation}
%
%
\begin{align}\label{Eq chi dependence}
& \boldsymbol{\chi} = \boldsymbol{\chi} (\xthilde) , &
& \boldsymbol{\chi}\cdot \n = -\vone \cdot \n \quad \text{ on } \Gamma .
\end{align}
\end{subequations}
In particular, notice that $ \partial_{z}\xi  =  \partial_{z}\xi(\xthilde) $.
\end{theorem}
\begin{proof}
Take $ \Phi = \left(0,\,\varphi^{2}\right)\in \Y $,  test \eqref{problem fixed geometry 2} and reordering the summands conveniently, we have
\begin{equation*}
\begin{split}
0 & =  \epsilon\int_{\Omega_{2}}\divt\vtaneps\,\varphi^{2} 
+\epsilon\int_{\Omega_{2}}\partial_{z}\vtaneps\cdot\gradt\zeta\,\varphi^{2} 
-\int_{\Omega_{2}}\partial_{z}\vtaneps\cdot\gradt\zeta\,\varphi^{2} 
+
\int_{\Omega_{2}}\partial_{z}\vnormeps\,\varphi^{2} 
\\
& =\int_{\Omega_{2}}\divt\left(\epsilon\,\vtaneps\right)\varphi^{2} 
+\int_{\Omega_{2}}\partial_{z}\big(\epsilon\,\vtaneps\big)\cdot\gradt\zeta\,\varphi^{2}
%
+\int_{\Omega_{2}}\big(\partial_{z}\,\vtaneps, \,\partial_{z}\,\vnormeps\big)\cdot\big(-\gradt\zeta,\,1\big)\,\varphi^{2}
\\
& =\int_{\Omega_{2}}\divt\left(\epsilon\,\vtaneps\right)\varphi^{2} 
+\int_{\Omega_{2}}\partial_{z}\left(\epsilon\,\vtaneps\right)\cdot\gradt\zeta\,\varphi^{2}
%
+\int_{\Omega_{2}}\big\vert\big(-\gradt\zeta,\,1\big)\big\vert\, \partial_{z}\left(\vepstwo\cdot\n\right)\varphi^{2} .
%
\end{split}
\end{equation*}
Letting $ \epsilon\,\downarrow\, 0 $ in the expression above we get
\begin{equation*}
\int_{\Omega_{2}}\divt\vtwo\varphi^{2}\ 
+\int_{\Omega_{2}}\partial_{z}\vtwo\cdot\gradt\zeta\,\varphi^{2} 
+\int_{\Omega_{2}}\big\vert\big(-\gradt\zeta,\,1\big)\big\vert
\, \partial_{z}\,\xi\,\varphi^{2} 
= 0 .
\end{equation*}
Recalling Equation \eqref{solutionconvergence} we have $ \partial_{z}\,\vtwo = 0 $; thus
\begin{equation*}
\int_{\Omega_{2}}\divt\vtwo\, \varphi^{2} 
+\int_{\Omega_{2}}\big\vert\big(-\gradt \zeta,\,1\big)\big\vert\partial_{z}\,\xi\,\varphi^{2}
= 0  .
\end{equation*}
Since the above holds for all $ \varphi^{2}\in L_{\,0}^{2}(\Omega_{2}) $ it follows that
\begin{equation*}
\divt\vtwo+
\vert(-\gradt\zeta,\,1)\vert\partial_{z} \,\xi = c ,
\end{equation*}
where $ c $ is a constant. In the expression above we observe that two out of three terms are independent from $ z $, then it follows that the third summand is also independent from $ z $. Since the vector $ (-\gradt\zeta,\,1) $ is independent from $ z $ we conclude that $ \partial_{z} \xi = \partial_{z} \xi (\xthilde) $ this, together with the boundary conditions \eqref{boundary conditions on normal velocity} yield the second equality in \eqref{partial xi dependence}.
\newline
\newline
Take $ \Psi = \big(\phi^{\,1},\ldots,\,\phi^{\,N}\big)\in \left(C_{\0}^{\infty}(\Omega_{2})\right)^{N} $ and for $ i = 1, 2, \ldots, N $, build the ``antiderivative" $ \varpi^{i} $ of $ \phi^{i} $ using the rule \eqref{negative antiderivative} and define $ \wtwo = \big(\varpi^{1},\ldots,\,\varpi^{\,N}\big) $. Use Lemma \ref{Th Surjectiveness from Hdiv to H^1/2} to construct $ \wone \in \Hdiv(\Omega_{1}) $ such that $ \wone\cdot\n = \wtwo\cdot\n $ on $ \Gamma $, $ \wone\cdot \n = 0 $ on $ \partial \Omega_{1} - \Gamma $ and 
\begin{equation}\label{Eq control of wone in testing for chi}
\big\Vert\wone\big\Vert_{\,\Hdiv(\Omega_{1})}\leq 
C \big\Vert\Psi\big\Vert_{\,\mathbf{L}^{2}(\Omega_{2})}.
\end{equation}
Therefore $ \w \defining \big(\wone,\,\wtwo\big)\in \X_{2} $; test \eqref{problem fixed geometry 1} with $ \w $ and regroup terms of higher order, we have
\begin{equation}\label{Eq passing the limit for chi}
\begin{split}
\int_{\Omega_1}  \Q\, \vepsone \cdot \w^1 
& -
\int_{\Omega_1}\pepsone \, \div\wone
\\
& -  \epsilon \int_{\Omega_2} \pepstwo \, \divt
\wtan 
- \epsilon\int_{\Omega_2} \pepstwo \, \partial_{z}\, \wtan\cdot\gradt\zeta 
%
+\int_{\Omega_2} \pepstwo \, \partial_{z}\, \wtan\cdot\gradt \zeta 
- \int_{\Omega_2} \pepstwo \, \partial_{z} \wnorm 
\\
& +\epsilon^{2}\int_{\Omega_2} \E\,\deps\vepstwo \bcol \deps\wtwo 
+\int_{\Omega_2} \E\,\partial_{z}\,\vepstwo\cdot\partial_{z}\,\wtwo 
\\
& +\alpha
\int_{\,\Gamma}\big(\,\vepsone\cdot\n\,\big)\,\big(\,\wone\,\cdot\n\,\big)\,dS
+
\,\epsilon^{\,2} \int_{\Gamma} \gamma \sqrt{\Q} \,\vtangeps \cdot
\wtang
 \,d S
 %
= \epsilon\,\int_{\Omega_2} {\f^{2,\epsilon}} \cdot
\wtwo . 
\end{split}
\end{equation}
The limit of all the terms in the expression above when $ \epsilon\,\downarrow\,0 $ is clear except for one term, which we discuss independently, i.e.
\begin{equation*}
\begin{split}
\epsilon^{2} \int_{\Omega_2} \E\,\deps\vepstwo  \bcol\deps\wtwo 
%
& = \epsilon^{2}\int_{\Omega_2} \E\,\deps\vepstwo \bcol\Big(\gradt\,\wtwo
+\partial_{z}\,\wtwo\gradt^{t}\zeta 
- \frac{1}{\epsilon}\,\partial_{z}\,\wtwo \gradt^{t}\zeta \Big) 
\\
& = \epsilon^{2}\int_{\Omega_2} \E\,\deps\vepstwo \bcol\big(\gradt\,\wtwo
+\partial_{z}\,\wtwo\gradt^{t}\zeta\big)
%
+\epsilon\int_{\Omega_2} \E\,\deps\vepstwo \bcol \,\partial_{z}\,\wtwo\big(-\gradt^{t}\,\zeta\big) . 
\end{split}
\end{equation*}
In the expression above, the first summand clearly tends to zero when $ \epsilon \downarrow 0 $. Therefore, we focus on the second summand
\begin{equation*}
\begin{split}
\epsilon\int_{\Omega_2} \E\,\deps\vepstwo \bcol \,\partial_{z}\,\wtwo\big(-\gradt^{t}\,\zeta\big) 
 = & \int_{\Omega_2} \E\,\big(\gradt\,\big(\epsilon\,\vepstwo\big)
+\partial_{z}\big(\epsilon\,\vepstwo\big)\gradt^{t}\zeta\big) \bcol \,\partial_{z}\,\wtwo\big(-\gradt^{t}\zeta\big) 
\\
& +\int_{\Omega_2} \E\,\partial_{z}\vepstwo\big(-\gradt^{t}\zeta\big) \bcol \,\partial_{z}\,\wtwo\big(-\gradt^{t}\zeta\big). 
\end{split}
\end{equation*}
All the terms in the left hand side can pass to the limit. Recalling the statement \eqref{Eq epsilon partial tangential L2 weak}, we conclude
\begin{equation*}
\epsilon^{2}\int_{\Omega_2} \E\,\deps\vepstwo \bcol\deps\wtwo 
\rightarrow \int_{\Omega_2} \E\,\gradt\,\vtwo
 \bcol \,\partial_{z}\,\wtwo\big(-\gradt^{t}\zeta\big) 
+\int_{\Omega_2} \E\,\boldsymbol{\chi}\big(-\gradt^{t}\,\zeta\big) \bcol 
\partial_{z}\,\wtwo\big(-\gradt^{t}\zeta\big) .
\end{equation*}
Letting $ \epsilon\downarrow 0 $ in \eqref{Eq passing the limit for chi} and considering the equality above we get
\begin{equation}\label{Eq passing the limit for chi 2}
\begin{split}
\int_{\Omega_1}   \Q\, \vone \cdot \wone 
-
& \int_{\Omega_1}  \pone \, \grad\cdot\wone
+\int_{\Omega_2} \ptwo \, \big(\partial_{z}\, \wtan,\,\partial_{z}\,\wnorm\big)
\cdot\big(\gradt\zeta, \,-1\big) 
+\int_{\Omega_2} \E\,\gradt\,\vtwo
 \bcol \,\partial_{z}\,\wtwo\big(-\gradt^{t}\zeta\big) 
 \\
& +\int_{\Omega_2} \E\,\boldsymbol{\chi}\big(-\gradt^{t}\zeta\big) \bcol 
\partial_{z}\,\wtwo\big(-\gradt^{t}\zeta\big) 
+\int_{\Omega_2} \E\,\boldsymbol{\chi}\cdot\partial_{z}\,\wtwo 
+\alpha
\int_{\,\Gamma}\big(\vone\cdot\n\,\big)\,\big(\wone\,\cdot\n\big)\,dS
= 0.
\end{split}
\end{equation}
We develop a simpler expression for the sum of the fourth, fifth and sixth terms
\begin{equation*}
\begin{split}
\int_{\Omega_2} \E\,\gradt\,\vtwo
 \bcol\,\partial_{z}\,\wtwo& \big(-\gradt^{t}\,\zeta\big) 
+\int_{\Omega_2} \E\,\boldsymbol{\chi}\big(-\gradt^{t}\,\zeta\big) \bcol 
\partial_{z}\,\wtwo\big(-\gradt^{t}\zeta\big) 
+\int_{\Omega_2} \E\,\boldsymbol{\chi}\cdot\partial_{z}\,\wtwo 
\\
& =\int_{\Omega_2} \E\,\big(\grad\vtwo\cdot(-\gradt\zeta,\,1)\big)
\cdot\partial_{z}\,\wtwo 
+\int_{\Omega_2} \E\,\boldsymbol{\chi}\, (-\gradt\zeta,\,1)^{t} \bcol \,\partial_{z}\,\wtwo(-\gradt\,\zeta,\,1)^{t} 
\\
& =\int_{\Omega_2} \E\vert(-\gradt\zeta,\,1)\vert\, \partial_{\,\n} \,\vtwo
\cdot\partial_{z}\,\wtwo 
+\int_{\Omega_2} \E\,\boldsymbol{\chi}\cdot\partial_{z}\wtwo\,(-\gradt\zeta,\,1)\cdot(-\gradt\zeta,\,1) .
\end{split}
\end{equation*}
Here $ \partial_{\,\n} $ is the normal derivative defined in the identity \eqref{Eq Derivative with respect to the normal}. We introduce the equality above in \eqref{Eq passing the limit for chi 2}, this yields
\begin{multline}\label{Eq basic limit in normal direction}
\int_{\Omega_1}  \Q\, \vone \cdot \wone 
-
\int_{\Omega_1}\pone \, \grad\cdot\wone
%
-\int_{\Omega_2} \vert(-\gradt\zeta, \,1)\vert\ptwo \,
 \partial_{z}\big(\wtwo\cdot\n\big) 
%
+\int_{\Omega_2} \E \, \vert(-\gradt\zeta,\,1)\vert\partial_{\,\n}\vtwo
\cdot\partial_{z}\,\wtwo 
\\
+\int_{\Omega_2} \E\,\boldsymbol{\chi}\cdot\partial_{z}\,\wtwo\big\vert\big(-\gradt\zeta,\,1\big)\big\vert^{2} 
+\alpha
\int_{\Gamma}\big(\vone\cdot\n\big)\,\big(\wone\,\cdot\n\big)\,d S
= 0 .
\end{multline}
%
%
Next, we integrate by parts the second summand in the first line, add it to the first summand and recall that $ \partial_{z}\wtwo = -\Psi $ by construction, thus
\begin{multline}\label{Eq chi limiting expression}
-\int_{\Gamma}\pone\, \big(\wone\cdot\n\big)\,d S
+\int_{\Omega_2} \vert(-\gradt\zeta, \,1)\vert\ptwo \; \n\cdot\Psi 
\\
-\int_{\Omega_2} \E\vert(-\gradt\zeta,\,1)\vert\partial_{\,\n}\vtwo
\cdot\Psi 
-\int_{\Omega_2} \E\,\boldsymbol{\chi}\cdot\Psi \, \vert(-\gradt\zeta,\,1)\vert^{2} 
%
+\alpha
\int_{\Gamma}\big(\vone\cdot\n\big)\,\big(\wone\,\cdot\n\big)\,dS
= 0 .
\end{multline}
In the expression above we develop the surface integrals as integrals over the projection $ G $ of $ \Gamma $ on $ \R^{N-1} $, we get
\begin{equation*}
-\int_{\Gamma}\pone\, \big(\wone\cdot\n\big)\,d S
 +\alpha
\int_{\Gamma}\big(\vone\cdot\n\big)\,\big(\wone\cdot\n\big)\,dS\\
=  \int_{G}\frac{1}{\n\cdot\eversor_{N}}\,\big[-\pone\big\vert_{\Gamma}
+\alpha\,\big(\vone\cdot\n\,\big\vert_{\Gamma}\big)\big]
\big(\wone\cdot\n\,\big\vert_{\Gamma}\big) \,d \xthilde.
\end{equation*}
Recalling that $ \displaystyle \wone\cdot\n = \int_{\zeta(\xthilde)}^{\zeta(\xthilde)+1}\Psi(\xthilde, z)\,dz\cdot\n $ on $ \Gamma $, the equality above transforms in
\begin{equation*}
\begin{split}
\int_{G}\frac{1}{\n\cdot\eversor_{N}}\,
\big[-\pone\big\vert_{\Gamma}
+\alpha\,\big(\vone\cdot\n\,\big\vert_{\Gamma}\big)\big]\,
& \int_{\,\zeta(\xthilde)}^{\zeta(\xthilde)+1}\Psi(\xthilde, z)\,dz\cdot\n(\xthilde) \,d \xthilde\\
& =  \int_{G}\int_{\,\zeta(\xthilde)}^{\zeta(\xthilde)+1}\frac{1}{\n\cdot\eversor_{N}}\,
\big[-\pone\big\vert_{\Gamma}
+\alpha\,\big(\vone\cdot\n\,\big\vert_{\Gamma}\big)\big](\xthilde)\, \n(\xthilde)\cdot
\Psi(\xthilde, z)\,dz \,d \xthilde\\
& = \int_{\Omega_{2}}\frac{1}{\n\cdot\eversor_{N}}
\,\big[-\pone\big\vert_{\Gamma}
+\alpha\,\big(\vone\cdot\n\,\big\vert_{\Gamma}\big)\big]\, \n\cdot\Psi \,dz \,d \xthilde .
\end{split}
\end{equation*}
Introducing the latter in \eqref{Eq chi limiting expression} we have
\begin{multline*}
\int_{\Omega_{2}}\frac{1}{\n\cdot\eversor_{N}}\,
\big[-\pone\big\vert_{\Gamma}
+\alpha\,\big(\vone\cdot\n\,\big\vert_{\Gamma}\big)\big]\n\cdot\Psi 
\,dz \,d \xthilde
+\int_{\Omega_2} \vert(-\gradt\zeta, \,1)\vert\ptwo \, \n\cdot\Psi \,d\xthilde\,dz \\
-\int_{\Omega_2} \E\,\vert(-\gradt\zeta,\,1)\vert\partial_{\,\n}\vtwo
\cdot\Psi\,d\xthilde \,dz 
-\int_{\Omega_2} \E\,\big\vert\big(-\gradt\zeta,\,1\big)\big\vert^{2} \, \boldsymbol{\chi}\cdot\Psi \,d\xthilde \,dz
= 0 .
\end{multline*}
Since the above holds for all $ \Psi\in \big(C_{0}^{\infty}(\Omega_{2})\big)^{N} $, we conclude
\begin{multline}\label{Eq higher order term chi identified}
\frac{1}{\n\cdot\eversor_{N}}\,
\big[-\pone\big\vert_{\Gamma}
+\alpha\,\big(\vone\cdot\n\,\big\vert_{\Gamma}\big)\big]\n
+ \vert(-\gradt\zeta, \,1)\vert\ptwo \, \n\\
-\E\,\vert(-\gradt\zeta,\,1)\vert\partial_{\,\n}\vtwo
- \E\,\vert(-\gradt \zeta,\,1)\vert^{2} \, \boldsymbol{\chi}
= 0\,,\quad\text{in}\;\mathbf{L}^{2}(\Omega_{2}) .
\end{multline}
%
%
In order to get the normal balance on the interface we could repeat the previous strategy but using $ \Psi\in C_{\,0}^{\,\infty}(\Omega_{2})^{N} $ such that $ \Psi = \big(\Psi\cdot\n\big)\, \n $, i.e. such that it is arallel to the normal direction. This would be equivalent to replace $ \Psi $ by $ \big(\Psi\cdot\n\big)\,\n $ in all the previous equations. Consequently, in order to get the normal balance, it suffices to apply $ \big(\cdot\dfrac{\n}{\vert(-\gradt \zeta,\,1)\vert^{2}} \big) $ to Equation \eqref{Eq higher order term chi identified}; such operation yields:
\begin{equation}\label{Eq balance of normal stress}
\frac{1}{\n\cdot\eversor_{N}}\,\frac{1}{\vert(-\gradt \zeta,\,1)\vert^{2}}
\,\big[-\pone\big\vert_{\Gamma}+\alpha\,\big(\vone\cdot\n\,\big\vert_{\Gamma}\big)\big]\\
+ \,\frac{1}{\vert(-\gradt \zeta,\,1)\vert}\,\ptwo \, 
-\E\,\frac{1}{\vert(-\gradt \zeta,\,1 ) \vert}\,\partial_{\,\n}\vtwo\cdot\n
- \E\,\partial_{z}\,\xi
= 0 .
\end{equation}
In the expression above the identity \eqref{Eq vtwo tangential and higher order terms relationship} has been used. Also notice that all the terms are independent from $z$, then the equation \eqref{Eq pressure two dependence} follows. Consequently, all the terms but the last in \eqref{Eq higher order term chi identified} are independent of $z$, therefore we conclude $ \boldsymbol{\chi} $ is independent from $ z $. Recalling \eqref{Eq vtwo tangential and higher order terms relationship} and \eqref{partial xi dependence} the second equality in 
\eqref{Eq chi dependence} follows and the proof is complete.
%
%
\qed
\end{proof}
\section{The Limiting Problem}\label{Sec Limiting Problem}
%
%
%
%
\noindent In this section we derive the form of the limiting problem and characterize it as a Darcy-Brinkman coupled system, where the Brinkman equation takes place in a $ N-1 $-dimensional manifold of $ \R^{N} $. First, we need to introduce some extra hypotheses to complete the analysis.
\begin{hypothesis}\label{Hyp Strong Convergence on the Forcing Terms}
In the following, it will be assumed that the sequence of forcing terms $ \big(\f^{2,\epsilon}: \epsilon > 0\big)\subseteq \mathbf{L}^{2}(\Omega_{2}) $ and $ \big(h^{1,\epsilon}: \epsilon > 0\big) \subseteq L^{2}(\Omega_{1}) $ are weakly convergent i.e., there exist $ \f^{2} \in \mathbf{L}^{2}(\Omega_{2}) $ and $ h^{1}\in L^{2}(\Omega_{1}) $ such that
%
\begin{align}\label{Ineq Strong Convergence on the Forcing Terms}
& \f^{2, \epsilon} \rightharpoonup \f^{2},&
&  h^{1, \epsilon} \rightharpoonup  h^{1}\, .
\end{align}
\end{hypothesis}
%
%
%
%
%
%
\subsection{The Tangential Behavior of the Limiting Problem}
%
%
%
%
\noindent Recalling
\eqref{solutionconvergence} and \eqref{Eq vtwo tangential and higher order terms relationship} clearly the lower order limiting velocity has the structure
\begin{equation}\label{Eq structure of tangential velocity}
\begin{Bmatrix}
\vtang\\
0
\end{Bmatrix} =
\begin{Bmatrix}
\vtang\pxthilde\\
0
\end{Bmatrix}.
\end{equation}
The above motivates the following definition.
\begin{definition}\label{Def lower order velocity space}
Let $ \xthilde \mapsto  U(\xthilde) $ be the matrix map introduced in Definition \ref{Def Local System of Coordinates}. Define the space $ \X_{\tang}\subseteq\X_2 $ by
%
%
%
\begin{equation} \label{Def space of tangential velocity 2}
\X_{\tang}\defining
\bigg\{\wtwo\in \X_2:\wtwo =
U(\xthilde)
\begin{Bmatrix}
\wtang\pxthilde\\
0
\end{Bmatrix}
\bigg\},
\end{equation}
%
endowed with the $ H^{1}(\Omega_{2}) $-norm.
\end{definition}
We have the following result
\begin{lemma}
The space $ \X_{\tang}\subset \X_{2} $ is closed.
\end{lemma}
\begin{proof}
Let $ \big(\wtwo\left(\ell\right):\ell\in \N\big)\subset \X_{\tang} $ and $ \wtwo\in \X_2 $ be such that $ \big\Vert\wtwo\big(\ell\big) - \wtwo\big\Vert_{1, \Omega_2}\,\rightarrow 0 $. We must show that $ \wtwo\in \X_{\tang} $. First notice that the convergence in $ \X_2 $ implies $ \big\Vert\wtwo\big(\ell\big) - \wtwo\big\Vert_{0,\Omega_2}\,\rightarrow 0 $. Recalling \eqref{Eq local and global velocities} and the fact that $ U(\xthilde) $ is orthogonal, we have
\begin{equation*}
\begin{bmatrix}
\UTtang  & \UTnormal \\
\UNtang  & \UNnormal
\end{bmatrix}^{t}\pxthilde
\begin{Bmatrix}
\wtan\left(\ell\right)\\[3pt]
\wnorm\left(\ell\right)
\end{Bmatrix}\pxthilde
=
\begin{Bmatrix}
\wtang\left(\ell\right)\\[3pt]
0
\end{Bmatrix}\pxthilde .
\end{equation*}
In the identity above, we notice that $ \wtan\big(\ell\big), \wtan\big(\ell\big) $ are convergent in the $ H^1 $-norm and the orthonormal matrix $ U $, has differentiability and boundedness properties. Therefore, we conclude that $ \wtang\big(\ell\big) $ is convergent in the $ H_1 $-norm, we denote the limit by $ \wtang =  \wtang\big(\xthilde, z\big) $. Now take the limit in the expression above in the $ L^2 $-sense; there are no derivatives involved, then we have
\begin{equation*}
\begin{bmatrix}
\UTtang  & \UTnormal \\
\UNtang  & \UNnormal
\end{bmatrix}^{t}\pxthilde
\begin{Bmatrix}
\wtan\\[3pt]
\wnorm
\end{Bmatrix}\pxthilde
=
\begin{Bmatrix}
\wtang\big(\xthilde, z\big)\\[3pt]
0
\end{Bmatrix} .
\end{equation*}
Observe that the latter expression implicitly states that $ \wtang = \wtang\pxthilde $. Finally, applying the inverse matrix again we have
\begin{equation*}
\wtwo = 
\begin{Bmatrix}
\wtan\\[3pt]
\wnorm
\end{Bmatrix}\pxthilde
=
\begin{bmatrix}
\UTtang  & \UTnormal \\
\UNtang  & \UNnormal
\end{bmatrix}\pxthilde
\begin{Bmatrix}
\wtang\\[3pt]
0
\end{Bmatrix}(\xthilde),
\end{equation*}
where the equality is in the $ L^2 $-sense. But we know that $ \wtang\in \big[H^1\big(\Omega_2\big)\big]^{N-1} $, therefore the equality holds in the $ H^1 $-sense too, i.e. $ \X_{\tang} $ is closed as desired.
\qed
\end{proof}
Next, we use space $ \X_{\tang} $ to determine the limiting problem in the tangential direction.
\begin{lemma}[Limiting Tangential Behavior's Variational Statement]
\label{Th Tangential Behavior of the Limit Problem}
Let $ \vtwo $ 
be the limit found in Theorem \ref{Th Direct Weak Convergence of Velocities} (ii). Then the following weak variational statement is satisfied
\begin{equation}  \label{Eq tangential limit problem with chi}
-   \int_{\Omega_2} \ptwo \, \divt
\wtan 
+ \int_{\Omega_2} \E\,\gradt\vtwo \bcol\gradt\wtwo 
%
%
+ \int_{\Gamma} \beta \sqrt{\Q} \,\vtang \cdot
\wtang
 \,d S 
= \int_{\Omega_2} {\mathbf f_{\tang}^{\,2}} \cdot
\wtang  
\quad \text{ for all } \wtwo \in \X_{\tang} .
\end{equation}
\end{lemma}
%
%
\begin{proof}
Let $ \wtwo\in \X_{\tang} $, then $ \left(0,\wtwo\right)\in \X $, test \eqref{problem fixed geometry 1} and get
\begin{equation*}
-  \epsilon \int_{\Omega_2} \pepstwo \, \divt
\wtan 
+\epsilon^{2}\int_{\Omega_2} \E\,\deps\vepstwo \bcol\deps\wtwo 
+\,\epsilon^{2} \int_{\Gamma} \beta \sqrt{\Q} \,\vtangeps \cdot
\wtang
 \,d S
= \epsilon\,\int_{\Omega_2} {\mathbf f_{\tang}^{2,\epsilon}} \cdot
\end{equation*}
Divide the whole expression over $ \epsilon $, expand the second summand according to the identity \eqref{Def operator D epsilon} and recall that $\partial_{z}\,\wtwo = 0$; this gives
\begin{multline*}
-   \int_{\Omega_2} \pepstwo \, \divt
\wtan 
+ \int_{\Omega_2} \E\,\left[\gradt\left(\epsilon\,\vepstwo\right)
+(\epsilon - 1)\partial_{z}\,\vepstwo\gradt^{t}\zeta \right] \bcol\gradt \, \wtwo 
%
+ \int_{\Gamma} \beta \sqrt{\Q} \,\epsilon\,\vtangeps \cdot
\wtang
 \,d S
= \int_{\Omega_2} {\mathbf f_{\tang}^{2,\epsilon}} \cdot
\wtang .
\end{multline*}
Letting $ \epsilon \downarrow 0 $, the limit $ \vtwo $ meets the condition
\begin{equation}\label{Eq tangential limit with chi first version}
-   \int_{\Omega_2} \ptwo \, \divt
\wtan 
+ \int_{\Omega_2} \E\,\left[\gradt \vtwo-\boldsymbol{\chi}\gradt^{t}\zeta\right] \bcol\gradt\wtwo 
+ \int_{\Gamma} \beta \sqrt{\Q} \,\vtang \cdot
\wtang
 \,d S
= \int_{\Omega_2} {\mathbf f_{\,\tang}^{\,2}} \cdot
\wtang . 
\end{equation}
We modify the higher order term using that $ \partial_{z}\,\wtwo = 0 $ 
%
%
\begin{equation*}
%
%
-\int_{\Omega_2} \E\,\boldsymbol{\chi}\gradt^{t}\zeta \bcol \gradt \wtwo
%
= \int_{\Omega_2} \E\,\boldsymbol{\chi}(-\gradt^{t}\zeta,\,1)\bcol\grad\wtwo
%
= \int_{\Omega_2} \E\,\vert(-\gradt\zeta,\,1)\vert\boldsymbol{\chi} 
\cdot\left(\grad\wtwo\cdot\n\right) 
\\
%
%
\end{equation*}
%
%
Recall that $ \wtwo \cdot \n = 0 $ because $ \wtwo \in \X_{\tang} $, then $ \partial_{\n} \wtwo = \grad \wtwo \cdot \n = 0 $.
Replacing the above expression in \eqref{Eq tangential limit with chi first version}, the statement \eqref{Eq tangential limit problem with chi} follows, since all the previous reasoning is valid for $ \wtwo \in \X_{\tang} $ arbitrary. 
\qed
\end{proof}
%
%
%
%
%
%
\subsection{The Higher Order Effects and the Limiting Problem}
%
%
%
%
\noindent 
The higher order effects of the $ \epsilon $-problem have to be modeled in the adequate space, to that end we use the information attained. We know the higher order term $ \boldsymbol{\chi} $ satisfy the condition \eqref{Eq chi dependence} and it belongs to $ \mathbf{L}^{2}(\Omega_{2}) $. This motivates the following definition
\begin{definition}
Define 
\begin{enumerate}[(i)]
\item The subspace 
\begin{equation}\label{Eq subsspace of higher order effects}
\W_{\high}\defining\big\{[\wone, \, \boldsymbol{\eta}] \in  \X
 : 
\boldsymbol{\eta}_{\tang} = \boldsymbol{0}_{T}, \,
\boldsymbol{\eta}\cdot \n = \wone\cdot\n
\big\vert_{\Gamma}(\xthilde)(1 - z)\big\} ,
\end{equation}
endowed with its natural norm. 

\item The space of limit normal effects in the following way
\begin{subequations}\label{Def space of higher order effects}
\begin{equation}\label{Eq space of higher order effects}
\X_{\high}^{0}\defining\big\{[\wone, \boldsymbol{\eta}] \in  \Hdiv(\Omega_{1})
\times\Hboldpartial
: \boldsymbol{\eta}_{\tang} = \boldsymbol{0}_{T}, \,\partial_{z}\boldsymbol{\eta} = 0,\,\boldsymbol{\eta}\cdot\n = -
\wone\cdot\n\big\vert_{\Gamma}(\xthilde)(1 - z)\big\} ,
\end{equation}
endowed with its natural norm
\begin{equation}\label{Eq space of higher order effects norm}
\big\Vert\big[\wone,\,\boldsymbol{\eta} \big]\big\Vert_{\,\X_{\high}^{0}}^{\,2} \defining
\big\Vert \wone\big\Vert_{\Hdiv(\Omega_{1})}^{\,2} 
+ \big\Vert \boldsymbol{\eta}\big\Vert_{\Hboldpartial}^{\,2}  .
\end{equation}
\end{subequations}
\end{enumerate}
\end{definition}
\begin{remark}\label{Rem Normal Effects Space}
\begin{enumerate}[(i)]

\item It is direct to prove that $ \X_{\high}^{0} $ is closed. 

\item Observe that due to its structure, the component $ \boldsymbol{\eta} $ of an element in $ \X_{\high}^{0} $ can be completely described by its normal trace on $ \Gamma $ i.e., the norm 
\begin{equation}\label{Eq space of higher order effects norm dimensional}
\big\Vert\big[\wone,\,\boldsymbol{\eta} \big]\big\Vert_{\,\X_{\high}^{0}}^{\,2} \defining
\big\Vert \wone\big\Vert_{\Hdiv(\Omega_{1})}^{\,2} 
+ \big\Vert \boldsymbol{\eta}\cdot \n \big\Vert_{0, \Gamma}^{\,2}  ,
\end{equation}
is equivalent to the norm \eqref{Eq space of higher order effects norm}. This feature will permit the dimensional reduction of the limiting problem formulation later on, see Section \ref{Sec Dimensional Reduction of the Limiting Problem}.

\item Let $ \vone $ and $ \xi $ be the limits found in the statements \eqref{velocity omega 1} and \eqref{convergence of normal velocity}, respectively. Define the function
\begin{equation}\label{Def Auxilliary Representation Normal Velocity}
\boldsymbol{\xi} \defining 
U\begin{Bmatrix}
\boldsymbol{0}_{T} \\[3pt]
\xi
\end{Bmatrix},
\end{equation}
then $ [\vone, \boldsymbol{\xi}] $ belongs to $ \X_{\high}^{0} $. This was one of the motivations behind the definition of $ \X_{\high}^{0} $ above. 


\item The information about the higher order term $ \boldsymbol{\chi} $ is complete only in its normal direction $ \boldsymbol{\chi}(\n) $. Furthermore, the facts that $ \boldsymbol{\chi} $ depends only on $ \xthilde$ (see Equation \eqref{Eq chi dependence}) and that $ \boldsymbol{\chi}\cdot \n = \partial_{z} \xi  =  - \vone\cdot \n\big\vert_{\Gamma}$, show that only information corresponding to the normal component of $ \boldsymbol{\chi} $ will be preserved by the modeling space $ \X_{\high}^{0} $, while the tangential component of the higher order term $ \boldsymbol{\chi}(\tang) $ will be given away for good. It is also observed that most of the terms involving the presence of $ \boldsymbol{\chi} $, require only its normal component, e.g. $ \boldsymbol{\chi} \cdot \partial_{\,\n} \wtwo = \boldsymbol{\chi}(\n) \cdot \partial_{\,\n} \wtwo $ in the third summand of the variational statement \eqref{Eq tangential limit problem with chi}. This was the reason why the space $ \X_{\high}^{0}  $ excludes tangential effects of the higher order term.
\end{enumerate}
\end{remark}
Before characterizing the asymptotic behavior of the normal flux we need a technical lemma
\begin{lemma}\label{Th Density Result}
The subspace $ \W_{\high} \subseteq \X $ is dense in $ \X^0_{\high} $.
\end{lemma}
\begin{proof} 
Consider an element $ \w = (\wone, \boldsymbol{\eta})\in \X^0_{\high} $, then $ \boldsymbol{\eta}_{\tang} = \boldsymbol{0}_{T} $, and $ \boldsymbol{\eta}\cdot\n  \in \Hpartial $ is completely
defined by its trace on the interface $\Gamma$. Given $ \epsilon>0 $,
take $\varpi \in H^{1}_0(\Gamma)$ such that $ \Vert \varpi -
\boldsymbol{\eta} \cdot \n \big\vert_{\,\Gamma} \Vert_{L^{2}(\Gamma)}\leq \epsilon $. Now extend
the function to the whole domain using the rule $ \varrho(\xthilde,
z)\defining\varpi(\xthilde)(1-z) $, then $ \Vert \varrho - \boldsymbol{\eta}\cdot \n
\Vert_{\Hpartial}\leq \epsilon $. 
From the construction of $ \varrho $ we know
that $ \Vert \varrho\big\vert_{\,\Gamma} - \boldsymbol{\eta}\cdot \n\big\vert_{\,\Gamma}
\Vert_{0,\Gamma} = \Vert \varpi - \boldsymbol{\eta}\cdot \n\big\vert_{\,\Gamma}
\Vert_{0,\Gamma}\leq \epsilon $. Define $ g = \varrho\big\vert_{\,\Gamma}
- \boldsymbol{\eta}\cdot \n\big\vert_{\,\Gamma}\in L^{2}(\Gamma) $, due to Lemma \ref{Eq
Normal Trace Definition} there exists $ \u\in \Hdiv(\Omega_{1}) $ such
that $ \u \cdot\n = g $ on $ \Gamma$, $\u\cdot \n = 0 $ on $ \partial
\Omega_{1}-\Gamma $ and $ \Vert \u \Vert_{\,\Hdiv (\Omega_1)}\leq C_1
\Vert g \Vert_{0,\Gamma} $ with $ C_{1} $ depending only on
$ \Omega_{1} $. Then, the function $ \wone + \u $ is such that
$ (\wone+\u)\cdot\n = \wone\cdot\n+\varpi - \boldsymbol{\eta}\cdot \n = \varpi $ and $\Vert
\wone+\u - \wone \Vert_{\,\Hdiv(\Omega_1)} = \Vert\u
\Vert_{\,\Hdiv(\Omega_1)}\leq C_1 \Vert g \Vert_{0,\Gamma}\leq
C_1\,\epsilon$. Moreover defining
\begin{equation*}
\wtwo \defining U \begin{Bmatrix}
\boldsymbol{0}_{T}\\[3pt]
 \varrho
\end{Bmatrix},
\end{equation*}
we notice that the function $ (\wone + \u,  \wtwo) $ belongs to $ \W_{\high} $. Due to the previous
observations we have
\begin{equation*}
\big\Vert \w - (\wone+\u, \, \wtwo)  \big\Vert_{\X^0} =
\big\Vert (\wone,\boldsymbol{\eta}) - (\wone+\u, \wtwo)
 \big\Vert_{\X^0_{\high}}\leq \sqrt{C_1 + 1 } \; \epsilon .
\end{equation*}
Since the constants depend only on the domains
$ \Omega_{1} $ and $ \Omega_{2} $, it follows that $ \W_{\high} $ is dense in
$ \X^0_{\high} $.
\qed
\end{proof}
\begin{lemma}[Limiting Normal Behavior's Variational Statement]
\label{Th Normal behavior of the Limit Problem}
Let $ \vone $, $ \vtwo $ be the limits found in Theorem \ref{Th Direct Weak Convergence of Velocities} and let $ \pone $, $ \ptwo $ be the limits found in Theorem \ref{Th Convergence of Pressure One}. Then, the following variational statement is satisfied
%
\begin{multline}\label{Eq higher order limit equation}
\int_{\Omega_1}  \Q\, \vone \cdot \w^1 \,d\x -
\int_{\Omega_1}\pone \, \div \wone
\,d\x
+ \int_{\Gamma} \ptwo \vert (-\gradt\zeta, 1) \vert (\wone\cdot\n\big\vert_{\Gamma})\, dS \\
+\int_{\Gamma} (\alpha + \Eav)\, (\vone\cdot\n) (\wone\cdot\n) \,dS
%
= 0,\quad\text{for}\;\text{all}\; 
\wone 
\in\X_{\high}^{0}.
\end{multline}
Here, it is understood that $ \displaystyle \Eav \defining \int_{\zeta(\xthilde)}^{\zeta(\xthilde)} \E \, dz $ i.e., $ \Eav $ is the average in the $ z $-direction.
\end{lemma}
\begin{proof}
Take $ [\wone,\boldsymbol{\eta}] \in \W_{\n} $, test \eqref{problem fixed geometry 1} and let $ \epsilon \rightarrow 0 $, this gives
\begin{multline} \label{Eq Testing on W_n and Passing to the Limit}
\int_{\Omega_1}  \Q\, \vone \cdot \w^1 \,d\x -
\int_{\Omega_1}\pone \, \div \wone
\,d\x 
%
+ \int_{\Omega_2} \ptwo \, \partial_{z}\, \boldsymbol{\eta}_{\scriptscriptstyle T}\cdot\gradt \zeta \,d\widetilde{\x} \,dz
- \int_{\Omega_2} \ptwo \, \partial_{z} \boldsymbol{\eta}_{\scriptscriptstyle N} \,d\widetilde{\x} \,dz \\
%
%
+\int_{\Omega_2} \E\,\boldsymbol{\chi}\cdot\partial_{z}\,\boldsymbol{\eta} \,d\widetilde{\x} \,dz
%
+\alpha
\int_{\Gamma}\big(\vone\cdot\n\big)\,\big(\wone\,\cdot\n\big)\,dS
%
= 0 .
\end{multline}
Notice that the third and fourth summands in the expression above can be written as
\begin{equation*}
\begin{split}
\int_{\Omega_2} \ptwo \, \partial_{z}\, \boldsymbol{\eta}_{\scriptscriptstyle T}\cdot\gradt \zeta \,d\widetilde{\x} \,dz
- \int_{\Omega_2} \ptwo \, \partial_{z} \boldsymbol{\eta}_{\scriptscriptstyle N} \,d\widetilde{\x} \,dz
= - \int_{\Omega_{2}} \ptwo \partial_{z} \boldsymbol{\eta} \cdot
\begin{Bmatrix}
- \gradt^{t}\zeta \\
1
\end{Bmatrix} 
& = - \int_{\Omega_2} \ptwo \vert (-\gradt\zeta, 1) \vert \, \partial_{z} \boldsymbol{\eta} \cdot \n \\
& = - \int_{\Omega_2} \ptwo \vert (-\gradt\zeta, 1) \vert (- \wone\cdot\n\big\vert_{\Gamma}) \\
& = - \int_{\Gamma} \ptwo \vert (-\gradt\zeta, 1) \vert (- \wone\cdot\n\big\vert_{\Gamma})\, dS, 
\end{split}
\end{equation*}
where the second equality holds by the definition of $ \W_{\n} $ and the last equality holds since $ \ptwo $ is independent from $ z $ \eqref{Eq pressure two dependence}. Next, recalling the identities \eqref{Eq vtwo tangential and higher order terms relationship}, \eqref{partial xi dependence} and \eqref{Eq chi dependence}, observe that 
\begin{equation*}
\begin{split}
\int_{\Omega_{2}} \E \, \boldsymbol{\chi} \cdot \partial_{z}\boldsymbol{\eta} 
= 
\int_{\Omega_{2}} \E \, (\boldsymbol{\chi}\cdot \n ) \, \partial_{z}(\boldsymbol{\eta} \cdot \n) 
= 
\int_{\Omega_{2}}  \E \, \partial_{z} \xi \, (- \wone\cdot\n\big\vert_{\Gamma}) 
&=
\int_{\Omega_{2}}  \E \, (- \vone\cdot\n\big\vert_{\Gamma}) \,  (- \wone\cdot\n\big\vert_{\Gamma}) \\
&=
\int_{\Gamma}  \Eav \, (- \vone\cdot\n\big\vert_{\Gamma}) \,  (- \wone\cdot\n\big\vert_{\Gamma}) \, dS .
\end{split}
\end{equation*}
Replacing the last two identities in \eqref{Eq Testing on W_n and Passing to the Limit} we conclude that the variational statement \eqref{Eq higher order limit equation} holds for every test function in $ \W_{\high} $. Since the bilinear form of the statement is continuous with respect to the norm $ \Vert \cdot \Vert_{\X^{0}_{\high}} $, it follows that the statement holds for all element $ \w \in \X_{\high}^{0} $. 
\qed
\end{proof}
%
%
%
%
%
%
%
%
\subsection{Variational Formulation of the Limit Problem}
\noindent 
In this section we give a variational formulation of the limiting problem and prove it is well-posed. We begin characterizing the limit form of the conservation laws
\begin{lemma}[Mass Conservation in the Limiting Problem]\label{Th Mass Conservation in the Limit Problem}
Let $ \vone, \vtwo $ be the limits found in Theorem \ref{Th Direct Weak Convergence of Velocities} then
\begin{subequations}\label{Eq limit conservation laws}
\begin{equation}  \label{Eq limit conservation law in Omega 1}
\div \vone = h^1 .
\end{equation}
\begin{align}\label{Eq limit conservation law in Omega 2}
\int_{\Omega_2}  \divt\vtwo\,\varphi^{2} 
-  \int_{\Gamma}  
\big\vert\big(-\gradt\zeta, 1\big)\big\vert\big(\vone\cdot\n\big)\varphi^{2}
\,d S
= 0, &
& \text{for all } \varphi^{2} \in L^{2}(\Omega_{2}), \, \varphi^{2} = \varphi^{2}(\xthilde) .
\end{align}
\end{subequations}
\end{lemma}
\begin{proof}
Take $ \Phi = \left(\varphi^1, 0\right)\in \Y $, test \eqref{problem fixed geometry 2} and let $ \epsilon\downarrow 0 $, we have
\begin{align*}  
& \int_{\Omega_1}\grad\cdot\vone
\varphi^1 
= \int_{\Omega_1} h^{1} \, \varphi^1 , &
& \text{ for all } \varphi^{1} \in L^{2}(\Omega_{1}).
\end{align*}
The statement above implies \eqref{Eq limit conservation law in Omega 1}.
\newline
\newline
For the variational statement \eqref{Eq limit conservation law in Omega 2}, first recall the dependence of the limit velocity given in equation \eqref{Eq pressure two dependence}. Hence, consider $ \Phi = \left(0, \varphi^2\right)\in \Y $ such that $ \varphi^2 = \varphi^2\pxthilde $, test \eqref{problem fixed geometry 2} and regroup terms using \eqref{Eq relation gradient surface and normal surface}, this yields
\begin{equation*}  
\int_{\Omega_2}  \divt\big(\epsilon\,\vtaneps\big)\,\varphi^{2} 
+ \int_{\Omega_2}  \partial_{z}\big(\epsilon\,\vtaneps\big)\cdot\gradt\zeta
\;\varphi^{2} 
+  \int_{\Omega_2}  \vert(-\gradt\zeta, 1)\vert\partial_{z}\big(\vepstwo\cdot\n\big)\varphi^2
= 0 .
\end{equation*}
Let $ \epsilon\downarrow 0 $ and get
\begin{equation*}  
\int_{\Omega_2}  \divt\vtwo\,\varphi^{2} 
+ \int_{\Omega_2}  \partial_{z}\vtwo\cdot\gradt\zeta
\;\varphi^{2} 
+  \int_{\Omega_2}  \vert(-\gradt\zeta, 1)\vert\partial_{z}\xi\;\varphi^{2}
\,d\widetilde{x} \,dz
= 0 .
\end{equation*}
In the expression above, recall that $ \partial_{z}\,\vtwo = 0 $, $ \varphi^{2} = \varphi_2(\xthilde) $ and the identity \eqref{partial xi dependence} then, the statement \eqref{Eq limit conservation law in Omega 2} follows.
\qed
\end{proof}
Next, we introduce the function spaces of the limiting problem
\begin{definition}\label{Def Limit Spaces}
Define the space of velocities by 
\begin{subequations}
\begin{equation}\label{Def Velocities Limit Space}
\X^{0} \defining \big\{ \w + \u : \w \in \X_{\high}^{0},\,  \u \in \X_{\tang} \big\}, 
\end{equation}
endowed with the natural norm of the space $ \X_{\high}^{0} \bigoplus \X_{\tang} $. Define the space of pressures by
\begin{equation}\label{Def Pressures Limit Space}
\Y^{0} \defining  
\big\{\Phi = \big(\varphi^1, \,\varphi^2\big)\in \Y:\varphi^2 = \varphi^2\pxthilde \big\} ,
\end{equation}
endowed with its natural norm.
\end{subequations}
\end{definition}
\begin{theorem}[Limiting Problem Variational Formulation]\label{Th Limiting Problem Variational Formulation}
Let $ \vone, \vtwo $ be the limits found in Theorem \ref{Th Direct Weak Convergence of Velocities} and let $ \pone $, $ \ptwo $ be the limits found in Theorem \ref{Th Convergence of Pressure One}. Then, they satisfy the following variational problem
\begin{subequations}\label{Eq variational formulation limit problem}
\begin{flushleft}
$ \big[\v,\,p\big]\in\X^{0}\times \Y^{0} :$
\end{flushleft}
\begin{multline}  \label{Eq variational formulation limit problem 1}
\int_{\Omega_{1}}\Q\,\vone\cdot\wone  
- \int_{\Omega_{1}}\pone\,\div\wone 
-  \int_{\Omega_2} \ptwo \, \divt
\wtan 
%
+ \int_{\Omega_2} \E\,\gradt\vtwo \bcol\gradt\wtwo 
%
\\
+ \int_{\Gamma} \beta \sqrt{\Q} \,\vtang \cdot \wtang
 \,d S
+ \int_{\Gamma} (\alpha + \Eav)\big(\vone\cdot\n\big)
\big(\wone\cdot\n\big)\,d S
%
+ \int_{\Gamma} 
\vert(-\gradt\zeta, \,1)\vert\ptwo \big(\wone\cdot\n\big) d S
%
%
= \int_{\Omega_2} {\mathbf f_{\,\tang}^{\,2}} \cdot
\wtang , 
\end{multline}
\begin{equation}  \label{Eq variational formulation limit problem 2}
\int_{\Omega_1}\div \vone
\varphi^1 
+\int_{\Omega_2}  \divt\vtwo\,\varphi^{2} 
-  \int_{\Gamma}  
\vert(-\gradt\zeta, 1)\vert 
\big(\vone\cdot\n\big)\varphi^{2}
\,d S 
= \int_{\Omega_1} h^{1} \, \varphi^1 ,
\end{equation}
\begin{flushright}
for all 
$ \big[\w,\,\Phi\big]\in\X^{0}\times \Y^{0} $.
\end{flushright}
\end{subequations}
%
%
Moreover, the problem \eqref{Eq variational formulation limit problem} is well-posed.
\end{theorem}
\begin{proof}
Since $ \left[\v,\,p\,\right] $ satisfies the variational statements \eqref{Eq tangential limit problem with chi}, \eqref{Eq higher order limit equation}, \eqref{Eq limit conservation law in Omega 1}, \eqref{Eq limit conservation law in Omega 2} as shown in Lemmas \ref{Th Tangential Behavior of the Limit Problem}, \ref{Th Normal behavior of the Limit Problem} \ref{Th Mass Conservation in the Limit Problem} respectively, it follows that $[\v,\,p\,]$ satisfies the problem \eqref{Eq variational formulation limit problem} above. 
\newline
\newline
In order to prove that the problem is well-posed we prove continuous dependence of the solution with respect to the data. Test \eqref{Eq variational formulation limit problem 1} with $ \big(\vone,\,\vtwo\big) $ and \eqref{Eq variational formulation limit problem 2} with $ \big(\pone,\,\ptwo\big) $, add them together and get
\begin{multline}\label{Eq limit problem on the diagonal}
\int_{\Omega_{1}}\Q\,\vone\cdot\vone 
+ \int_{\Omega_2} \E\,\gradt\vtwo \bcol\gradt\vtwo 
%
%
+ \int_{\Gamma} \beta \sqrt{\Q} \,\vtang \cdot
\vtang
 \,d S
+ \int_{\Gamma} (\alpha + \Eav)\big(\vone\cdot\n\big)
\big(\vone\cdot\n\big)\,d S
\\
= \int_{\Omega_2} {\mathbf f_{\,\tang}^{\,2}} \cdot
\vtang 
+ \int_{\Omega_{1}}h^{\,1}\,\pone . 
\end{multline}
Applying the Cauchy-Bunyakowsky-Schwarz inequality to the right hand side of the expression above and recalling that $ \vtang $ is constant in the $ z $-direction we get
\begin{equation}\label{Ineq Estimate Diagonal Limit}
\begin{split}
\int_{\Omega_2} {\mathbf f_{\,\tang}^{\,2}} \cdot
\vtang 
+ \int_{\Omega_{1}}h^{\,1}\,\pone  
&\leq 
\big\Vert{\mathbf f_{\,\tang}^{\,2}}\big\Vert_{0,\Omega_{2}}
\big\Vert\vtang\big\Vert_{0,\Omega_{2}}
+ \big\Vert h^{1}\big\Vert_{0,\Omega_{1}}
\big\Vert\pone\big\Vert_{0,\Omega_{1}} 
\\
& \leq 
\big\Vert{\mathbf f_{\,\tang}^{\,2}}\big\Vert_{0,\Omega_{2}}
\big\Vert\vtang\big\Vert_{0,\Gamma}
+ 
\widetilde{C} \, \big\Vert h^{1}\big\Vert_{0,\Omega_{1}}
\big\Vert\grad \pone\big\Vert_{0,\Omega_{1}} 
\\
& 
\leq 
\big\Vert{\mathbf f_{\,\tang}^{\,2}}\big\Vert_{0,\Omega_{2}}
\big\Vert\vtang\big\Vert_{0,\Gamma}
+ 
C\, \big\Vert h^{1}\big\Vert_{0,\Omega_{1}}
\big\Vert \Q \vone\big\Vert_{0,\Omega_{1}} \\
& \leq \widetilde{C} \Big[
\Big\Vert{\mathbf f_{\tang}^{\,2}}\big\Vert_{0,\Omega_{2}}^{\,2}
+ \big\Vert h^{1}\big\Vert_{0,\Omega_{1}}^{\,2}\Big]^{1/2}
\Big[\big\Vert\vtang\big\Vert_{0,\Gamma}^{\,2}
+ \big\Vert\vone\big\Vert_{0,\Omega_{1}}^{\,2}\Big]^{1/2} .
\end{split}
\end{equation}
Here, the second and third inequality holds because $ \pone $ satisfies respectively the drained boundary conditions (Poincar\'e's inequality applies) and Darcy's equation as stated in \eqref{convergence of the pressure in Omega_1}. Finally, the fourth inequality is a new application of the Cauchy-Bunyakowsky-Schwarz inequality for 2-D vectors. Introducing \eqref{Ineq Estimate Diagonal Limit} in \eqref{Eq limit problem on the diagonal} and recalling Hypothesis \eqref{Hyp Bounds on the Coefficients} on the coefficients $ \Q, \alpha, \beta $ and $ \E $ we have
\begin{equation}\label{Ineq First Continuity Estimate}
\Big[ 
\big\Vert\vone\big\Vert_{0,\Omega_{1}}^{\,2}
+ \big\Vert\vone\cdot \n\big\Vert_{\Gamma}^{\,2}
+ \big\Vert\gradt\vtang\big\Vert_{0,\Omega_{2}}^{\,2}
+ \big\Vert\vtang\big\Vert_{0,\Gamma}^{\,2}
\Big]^{1/2}
\leq \widetilde{C} \Big[
\Big\Vert{\mathbf f_{\tang}^{\,2}}\big\Vert_{0,\Omega_{2}}^{\,2} 
+ \big\Vert h^{1}\big\Vert_{0,\Omega_{1}}^{\,2}\Big]^{1/2}.
\end{equation}
Recalling \eqref{Eq velocity one conservation}, the expression above implies that 
\begin{equation}\label{Ineq Continuity Estimate vone}
\big\Vert\vone\big\Vert_{\Hdiv(\Omega_{1})}
\leq \widetilde{C} \Big[
\Big\Vert{\mathbf f_{\tang}^{\,2}}\big\Vert_{0,\Omega_{2}}^{\,2} 
+ \big\Vert h^{1}\big\Vert_{0,\Omega_{1}}^{\,2}\Big]^{1/2}.
\end{equation}
Next, recalling that $ \wtang $ is independent from $ z $ (see \eqref{solutionconvergence}), it follows that $ \big\Vert\vtang\big\Vert_{0,\Gamma} = \big\Vert\vtang\big\Vert_{0,\Omega_{2}} $ and that $ \big\Vert \grad \vtang\big\Vert_{0,\Omega_{2}} = \big\Vert \gradt \vtang\big\Vert_{0,\Omega_{2}} $. Therefore \eqref{Ineq First Continuity Estimate} yields
\begin{equation}\label{Ineq Continuity Estimate vtwo}
\big\Vert\vtwo\big\Vert_{1,\Omega_{2}}
\leq \widetilde{C} \Big[
\Big\Vert{\mathbf f_{\tang}^{\,2}}\big\Vert_{0,\Omega_{2}}^{\,2} 
+ \big\Vert h^{1}\big\Vert_{0,\Omega_{1}}^{\,2}\Big]^{1/2}.
\end{equation}
Again, recalling that $ \pone $ satisfies the Darcy's equation and the drained boundary conditions (Poincar\'e's inequality applies) as stated in \eqref{convergence of the pressure in Omega_1}, the estimate \eqref{Ineq Continuity Estimate vone} implies
\begin{equation}\label{Ineq Continuity Estimate pone}
\big\Vert\pone\big\Vert_{1, \Omega_{1}}
\leq \widetilde{C} \Big[
\Big\Vert{\mathbf f_{\tang}^{\,2}}\big\Vert_{0,\Omega_{2}}^{\,2} 
+ \big\Vert h^{1}\big\Vert_{0,\Omega_{1}}^{\,2}\Big]^{1/2}.
\end{equation}
Next, in order to prove continuous dependence for $\ptwo$ recall \eqref{Eq higher order term chi identified} where it is observed that all the terms are already continuously dependent on the data, then it follows that
\begin{equation}\label{Est Continuity Estimate pone}
\big\Vert \ptwo \big\Vert_{0,\Omega_{2}}
\leq C \Big[
\big\Vert{\mathbf f_{\tang}^{\,2}}\big\Vert_{0,\Omega_{2}}^{2}
+ \big\Vert h^{1}\big\Vert_{0,\Omega_{1}}^{\,2}\Big]^{1/2} .
\end{equation}
Finally, in order to prove the uniqueness of the solution, assume there are two solutions, test the problem \eqref{Eq variational formulation limit problem} with its difference and subtract them. We conclude the difference of solutions must satisfy the problem \eqref{Eq variational formulation limit problem} with null forcing terms which implies, due to \eqref{Ineq Continuity Estimate vone}, \eqref{Ineq Continuity Estimate vtwo} \eqref{Ineq Continuity Estimate pone} and \eqref{Est Continuity Estimate pone} that the difference of solutions is equal to zero, i.e. the solution is unique. Since \eqref{Eq variational formulation limit problem} has a solution, which is unique and it continuously depend on the data, it follows that the problem is well-posed.
\qed
\end{proof}
\begin{corollary}
The weak convergence statements in Corollaries \ref{Th Direct Weak Convergence of Velocities} and \ref{Th Convergence of Pressure One} hold for the whole sequence $ \big((\veps, \peps): \epsilon > 0\big) $ of solutions. 
\end{corollary}
\begin{proof}
It suffices to observe that due to Hypothesis \ref{Hyp Strong Convergence on the Forcing Terms} the limiting problem \eqref{Eq variational formulation limit problem} has unique forcing terms. Therefore, any subsequence of the solutions $ \big((\veps, \peps): \epsilon > 0\big) $ would have a weakly convergent subsequence, whose limit is the solution of problem \eqref{Eq variational formulation limit problem} $ (\v, p) $, which is also unique, due to Theorem \ref{Th Limiting Problem Variational Formulation}. Hence, the result follows.
\end{proof}
%
%
%
%
\section{Closing Remarks}\label{Sec Remarks Limiting Problem}
%
%
\noindent We finish the paper highlighting some aspects that were meticulously addressed in \cite{ShowMor17}.
%
%
%
\subsection{A Mixed Formulation for the Limiting Problem}\label{Sec Mixed Variational Formulation for Limit}
%
%
\noindent For an independent well-posedness proof of the problem \eqref{Eq variational formulation limit problem}, define the operators 
\begin{subequations} \label{Def limit operator A chi free}
\begin{align}\label{Def limit operator A 1 chi free}
& A^{0}:\X^{0}\rightarrow (\X^{0})',&
& A^{0}\defining 
\begin{bmatrix}
\Q+\gamma_{\n}'\big[ \alpha + \Eav \big]\,\gamma_{\n} &
0 \\
0
& \gamma_{\tang} ' \beta\,\sqrt{\Q}\, \gamma_{\tang}
- \divt \E \gradt  
\end{bmatrix}
\end{align}
and
\begin{align}\label{Def limit operator A 2 chi free}
& B^{0}:\X_{0}\rightarrow (\Y^{0})' , &
& B^{0}\defining 
\begin{bmatrix}
\div &  0\\
\gamma_{\tang}'\vert (-\gradt \zeta, 1) \vert \gamma_{\n}  & \divt
\end{bmatrix} .
\end{align}
\end{subequations}
Then, the variational formulation of the problem \eqref{Eq variational formulation limit problem} has the following mixed formulation
\begin{equation}\label{Def Limit Problem MIxed Formulation} 
\begin{split} 
%
[\,\v, p\,]\in \X^{0}\times\Y^{0}:
A^{0}\,\v- (B^{0}) '\,p &=\mathbf{f}^{2} , 
\\
%
 B^{0}\,\v &= h^{1} . 
%
\end{split}
\end{equation}
The proof now follows showing that the hypotheses of Theorem \ref{Th well posedeness mixed formulation classic} are satisfied; the strategy is completely analogous to that exposed in Lemma 17, Lemma 18 and Theorem 19 in \cite{ShowMor17}. 
%
%
%
%
%
\subsection{Dimensional Reduction of the Limiting Problem}\label{Sec Dimensional Reduction of the Limiting Problem}
%
%
\noindent It is direct to see that since $ \X_{\tang} $ and $ \Y^{0} $ do not change on the $ z $-direction inside $ \Omega_{2} $, the integrals on this domain can be reduced to integrals on the interface $ \Gamma $. This yields a problem coupled on $ \Omega_{1} \times \Gamma $ equivalent to \eqref{Eq variational formulation limit problem}. To that end  we introduce the following spaces:
%
\begin{subequations}
\begin{equation}\label{Eq space of higher order effects dimensional}
\X_{\high}^{00}\defining\big\{\wone \in  \Hdiv(\Omega_{1})
: \wone\cdot\n\big\vert_{\Gamma}\in L^{2}(\Gamma)\big\} ,
\end{equation}
endowed with the norm \eqref{Eq space of higher order effects norm dimensional} (clearly, $ \X_{\high}^{00} $ is isomorphic to $ \X^{0}_{\high} $ \eqref{Eq space of higher order effects}) and the space
\begin{equation} \label{Def space of tangential velocity 2 dimensional}
\X_{\tang}^{00}\defining
\big\{\wtwo\in \H1bold(\Gamma):
\wtwo\pxthilde \cdot \n\pxthilde = 0 \text{ for all } \xthilde \in G, 
\wtwo = 0 \text{ on } \partial \Gamma \big\} ,
\end{equation}
endowed with its natural norm. 
\end{subequations}
\begin{remark}[The spaces $ L^{2}(\Gamma)$ and $ H^{1}(\Gamma) $]\label{Rem Manifold H1 space}
Since $ \Gamma $ is a surface ($ \R^{N-1} $ manifold) as described by the identity \eqref{omega 2 epsilon}, it is completely characterized by its \textbf{global chart} $ \zeta: G \rightarrow \R  $. Therefore a function $ u : \Gamma \rightarrow \R $, $ \gamma \mapsto u(\gamma) $, can be seen as $ u_{G}: G \rightarrow \R  $, $ \xthilde \mapsto u(\xthilde, \zeta(\xthilde) ) $, with $ G $ being the orthogonal projection of the surface $ \Gamma $ into $ \R^{N -1} $. Identifying $ u $ with $ u_{G} $ allows to well-define integrability and differentiability. Hence, the space $ L^{2} (\Gamma) $ is characterized by the equality: $ \int_{\Gamma} u^{2} dS = \int_{G} u_{G}^{2} \vert (\grad\zeta, 1) \vert d\xthilde $, where $ d\xthilde $ is the Lebesgue measure in $ G \subseteq \R^{N-1} $. In the same fashion, the space $ H^{1} (\Gamma) $ is the closure of the $ C^{1} (\Gamma) $ space in the natural norm $ \Vert u \Vert_{0, \Gamma}^{2} \defining \Vert u \Vert_{0, \Gamma}^{2} + \Vert \gradt u \Vert_{0, \Gamma}^{2} $. (Clearly, $ \gradt $ suffices to store all the differential variation of a function $ u : \Gamma \rightarrow \R $.)
\end{remark}
With the definitions above, define the space of velocities
\begin{subequations}
\begin{equation}\label{Def Velocities Limit Space dimensional}
\X^{00} \defining \big\{ \w + \u : \w \in \X_{\high}^{00},\,  \u \in \X_{\tang}^{00} \big\}, 
\end{equation}
endowed with the natural norm of he space $ \X_{\high}^{0)} \bigoplus \X_{\tang}^{00} $. Next, define the space of pressures by
\begin{equation}\label{Def Pressures Limit Space dimensional}
\Y^{00} \defining  
L^{2}(\Omega_{1})\times L^{2}(\Gamma),
\end{equation}
endowed with its natural norm.
\end{subequations}
Therefore,  the problem \eqref{Eq variational formulation limit problem} is equivalent to
\begin{subequations}\label{Eq variational formulation limit problem dimensional}
\begin{flushleft}
$ \big[\v,\,p\big]\in\X^{00}\times \Y^{00} :$
\end{flushleft}
\begin{multline}  \label{Eq variational formulation limit problem 1 dimensional}
\int_{\Omega_{1}}\Q\,\vone\cdot\wone  
- \int_{\Omega_{1}}\pone\,\div\wone 
-  \int_{\Gamma} \ptwo \, \divt
\wtan 
%
+ \int_{\Gamma} \Eav\,\gradt\vtwo \bcol\gradt\wtwo 
%
%
+ \int_{\Gamma} \beta \sqrt{\Q} \,\vtwo \cdot \wtwo
 \,d S \\
+ \int_{\Gamma} (\alpha + \Eav)\big(\vone\cdot\n\big)
\big(\wone\cdot\n\big)\,d S
%
+ \int_{\Gamma} 
\vert(-\gradt\zeta, \,1)\vert\ptwo \big(\wone\cdot\n\big) d S
%
%
= \int_{\Gamma} \bar{{\mathbf f^{2}} }\cdot
\wtwo , 
\end{multline}
\begin{equation}  \label{Eq variational formulation limit problem 2 dimensional}
\int_{\Omega_1}\div \vone
\varphi^1 
+\int_{\Gamma}  \divt\vtwo\,\varphi^{2} 
-  \int_{\Gamma}  
\vert(-\gradt\zeta, 1)\vert 
\big(\vone\cdot\n\big)\varphi^{2}
\,d S 
= \int_{\Omega_1} h^{1} \, \varphi^1 ,
\end{equation}
\begin{flushright}
for all 
$ \big[\w,\,\Phi\big]\in\X^{00}\times \Y^{00} $,
\end{flushright}
\end{subequations}
where $ \bar{{\mathbf f^{2}} }(\xthilde) \defining \int_{\zeta(\xthilde)}^{\zeta(\xthilde) + 1} {\mathbf f^{2}} dz $.
\begin{remark}[The Brinkman Equation]\label{Rem Brinkman Equation}
Notice that in the equation \eqref{Eq variational formulation limit problem 1 dimensional} the product $ \vtang \cdot \wtang $ has been replaced by $ \vtwo \cdot \wtwo $ (for consistency $ \bar{{\mathbf f_{\tang}^{2}} }\cdot\wtang $ was replaced by $ \bar{{\mathbf f^{2}} }\cdot
\wtwo $). This is done in order to attain a Brinkman-type form in the third, fourth and fifth summands of equation \eqref{Eq variational formulation limit problem 1 dimensional}. Also notice that although $ \vtwo\cdot \n = 0 $ and $ \wtwo\cdot \n = 0 $ the product $ \gradt\vtwo \bcol\gradt\wtwo $ can not be replaced by $ \gradt\vtang \bcol\gradt\wtang $, due to the differential operators (the orthogonal matrix $ U $ depends on $ \xthilde $). This is why we give up expressing the activity on the interface $ \Gamma $ exclusively in terms of tangential vectors, as its is natural to look for.
\end{remark}
%
%
%
%
%
\subsection{Strong Convergence of the Solutions}\label{Sec Strong Convergence}
%
%
\noindent In contrast to the asymptotic analysis in \cite{ShowMor17}, the strong convergence of the solutions can not be concluded. The main reason is the presence of the higher order term $ \boldsymbol{\chi} $, weak limit of $ \partial_{z} \vepstwo $. As it can be seen in the proof of Theorem \ref{Th Tangential Behavior of the Limit Problem}, 
the higher order term $ \boldsymbol{\chi} $ can be removed because the quantifier $ \wtwo $ belongs to $ \X_{\tang} $. However, when testing the problem \eqref{problem fixed geometry} on the diagonal $ [\veps, \peps] $ and adding the equations to get rid of the mixed terms, the quantifier $ \vepstwo $ does not belong to $ \X_{\tang} $. As a consequence, the terms $ \Vert \sqrt{\E}\, \deps (\epsilon\vepstwo)\Vert_{0, \Omega_{2}}^{2} + 
\Vert \E\partial_{z} \vepstwo \Vert_{0, \Omega_{2}} $ contain in its internal structure, inner products of the type 
\begin{equation}\label{Eq Crossed Higher vs Lower terms}
\int_{\Omega_{2}}\E \partial_{z} \vepstwo 
\begin{Bmatrix}
-\grad \zeta \\
1
\end{Bmatrix}\bcol
\grad (\epsilon \, \vepstwo) 
= 
\int_{\Omega_{2}}\E 
\vert(-\grad \zeta ,1) \vert 
\partial_{z} \vepstwo\cdot
\grad (\epsilon \, \vepstwo)  \cdot \n ,
\end{equation}  
which can not be combined/balanced with other terms present in the evaluation of the diagonal. The product above is not guaranteed to pass to the limit $ \int_{\Omega_{2}}\E \vert(-\grad \zeta ,1) \vert 
\boldsymbol{\chi}
\cdot
\grad \vtwo\cdot \n $, because both factors are known to converge weakly, but none has been proved to converge strongly. Such convergence would be ideal since $ \vtwo\in \X_{\tang} $, therefore $ \partial_{\n} \vtwo = \grad \vtwo \cdot \n = 0 $ and the term \eqref{Eq Crossed Higher vs Lower terms} would converge to zero. The latter would yield the strong convergence of the norms for $ \Vert \gradt (\epsilon\vepstwo)\Vert_{0, \Omega_{2}} $ and $ \Vert \partial_{z}\vepstwo \Vert_{0, \Omega_{2}} $ and the desired strong convergence should follow. 
\newline

\noindent More specifically, the surface geometry states that the normal ($ \n $) and the tangential directions ($ \tang $) are the important ones, around which the information should be arranged. On the other hand, the estimates yield its information in terms of $ \xthilde $ ($ T $) and $ z $ ($ N $). Such disagreement has the effect of keeping intertwined the higher order and lower order terms to the extent of allowing to conclude weak, but not strong convergence.
%
%
%
%
\subsection{Ratio of Velocities}
%
%
\noindent The relationship of the velocity in the tangential direction with respect to the velocity in the normal direction is very high and tends to infinity as expected for most of the cases. We know that $\big(\Vert\,\vnormaleps\,\Vert_{0,\Omega_2}:\epsilon>0\big)$ is bounded, therefore
$\Vert\,\epsilon\,\vnormaleps\,\Vert_{0,\Omega_2}=\epsilon\,\Vert\,\vnormaleps\,\Vert_{0,\Omega_2}\rightarrow
0$. Suppose first that $\vtang\neq0$ and consider the
ratios:
\begin{equation*}
\frac{\Vert\,\vtangeps\,\Vert_{0,\,\Omega_2}}{\Vert\,\vnormaleps\,\Vert_{0,\,\Omega_2}}=
\frac{\Vert\,\epsilon\,\vtangeps\,\Vert_{0,\,\Omega_2}}{\Vert\,\epsilon\,\vnormaleps\,\Vert_{0,\,\Omega_2}}
\geq 
\frac{\liminf\Vert\,\epsilon\,\vtangeps\,\Vert_{0,\,\Omega_2}}{\Vert\,\epsilon\,\vnormaleps\,\Vert_{0,\,\Omega_2}}
>
\frac{\Vert\,\vtang\,\Vert_{0,\,\Omega_2}-\delta}
{\Vert\,\epsilon\,\vnormaleps\,\Vert_{0,\,\Omega_2}}>0 .
\end{equation*}
The lower bound holds true for $ \epsilon>0 $ small enough and adequate $ \delta>0 $, then we conclude the quotient of tangent component over normal component $ L^{2} $-norms blows-up to infinity, i.e. the tangential velocity is much faster than the normal one in the thin channel. 
\newline

\noindent If, on the other hand $ \vtang=0 $ nothing can be concluded, since it can not be claimed that $ \vone\cdot\n =0 $ on $ \Gamma $ unless $ \f^{2} =0 $ is enforced, trivializing the activity on $ \Omega_2 $. Therefore, it can only be concluded that $ \big\Vert\vtangeps\big\Vert_{0,\Omega_2}\gg \big\Vert\vnormaleps\big\Vert_{0,\Omega_2} $ for $ \epsilon>0 $ small enough, when $ \vtang \neq 0 $, as discussed above.
%
\section*{Acknowledgements}
%
%
\noindent The Author wishes to acknowledge Universidad Nacional de Colombia, Sede Medell\'in for its support in this work through the project HERMES 27798. The Author also wishes to thank his former PhD adviser, Professor Ralph Showalter from Oregon State University, who trained him in the field of multiscale PDE analysis. Special thanks to Professor Ma\l{}gorzata Peszy\'nska from Oregon State University who was the first to challenge the Author in analyzing curved interfaces and suggested potential techniques and scenarios to address the problem. 
%
%
%
%
%
%
\bibliographystyle{plain}
\def\cprime{$'$} \def\cprime{$'$} \def\cprime{$'$}

\end{document}